\documentclass[a4paper,12pt,intlimits,oneside]{amsart}

\usepackage{amsmath}
\usepackage{amssymb}
\usepackage{amsfonts}
\usepackage{amsthm,amscd}
\usepackage{amsthm,amscd}
\usepackage{geometry}

\newtheorem{thm}{Theorem}[section]
\newtheorem{prop}[thm]{Proposition}
\newtheorem{cor}[thm]{Corollary}

\newtheorem{defn}[thm]{Definition}
\newtheorem{remark}[thm]{Remark}

\theoremstyle{definition}
\newcommand{\comment}[1]{}

\numberwithin{equation}{section}

\def\lsim{\raisebox{-1ex}{$~\stackrel{\textstyle <}{\sim}~$}}

\theoremstyle{definition}

\begin{document}
\title[Boundedness of some classical linear operators]{Some applications of the dual spaces of Hardy-amalgam spaces}
\author[Z.V.P. Abl\'e]{Zobo Vincent de Paul Abl\'e}
\address{Laboratoire de Math\'ematiques Fondamentales, UFR Math\'ematiques et Informatique, Universit\'e F\'elix Houphou\"et-Boigny Abidjan-Cocody, 22 B.P 582 Abidjan 22. C\^ote d'Ivoire}
\email{{\tt vincentdepaulzobo@yahoo.fr}}
\author[J. Feuto]{Justin Feuto}
\address{Laboratoire de Math\'ematiques Fondamentales, UFR Math\'ematiques et Informatique, Universit\'e F\'elix Houphou\"et-Boigny Abidjan-Cocody, 22 B.P 1194 Abidjan 22. C\^ote d'Ivoire}
\email{{\tt justfeuto@yahoo.fr}}



\date{}

\begin{abstract}
In this paper, thanks to the generalizations of the dual spaces of the Hardy-amalgam spaces $\mathcal H^{(q,p)}$ and $\mathcal{H}_{\mathrm{loc}}^{(q,p)}$ for $0<q\leq1$ and $q\leq p<\infty$, obtained in our earlier paper \cite{AbFt3}, we prove that the inclusion of $\mathcal H^{(1,p)}$ in $(L^1,\ell^p)$ for $1\leq p<\infty$ is strict, and more generally, the one of $\mathcal H^{(q,p)}$ in $\mathcal{H}_{\mathrm{loc}}^{(q,p)}$ for $0<q\leq1$ and $q\leq p<\infty$. Moreover, as other applications, we obtain results of boundedness of Calder\'on-Zygmund and convolution operators, generalizing those known in the context of the spaces $\mathcal H^1$ and $BMO(\mathbb{R}^d)$. 
\end{abstract}

\maketitle

\section{Introduction}
Let $\varphi\in\mathcal C^\infty(\mathbb R^d)$ with support in $B(0,1)$ such that $\int_{\mathbb R^d}\varphi(x)dx=1$, where $B(0,1)$ is the unit open ball centered at $0$ and $\mathcal C^\infty(\mathbb R^d)$ denotes the space of infinitely differentiable complex valued functions on $\mathbb R^d$. The Hardy-amalgam spaces $\mathcal H^{(q,p)}$ and $\mathcal{H}_{\mathrm{loc}}^{(q,p)}$ ($0<q,p<\infty$), introduced in \cite{AbFt}, are a generalization of the classical Hardy spaces $\mathcal H^q$ and $\mathcal{H}_{\mathrm{loc}}^q$ in the sense that they are respectively the spaces of tempered distributions $f$ such that the maximal functions
\begin{equation}
\mathcal M_{\varphi}(f):=\sup_{t>0}|f\ast\varphi_t|\ \ \ \text{ and }\ \ \ {\mathcal{M}_{\mathrm{loc}}}_{_{\varphi}}(f):=\sup_{0<t\leq 1}|f\ast\varphi_t|\ \ \ \ \ \label{maximal}
\end{equation}
belong to the Wiener amalgam spaces $(L^q,\ell^p):=(L^q,\ell^p)(\mathbb R^d)$, where $\varphi_t(x)=t^{-d}\varphi(t^{-1}x),\, t>0$ and $x\in\mathbb R^d$. 

We recall that for $0<p,q\leq\infty$, a locally integrable function $f$ belongs to the amalgam space $(L^q,\ell^p)$ if $$\left\|f\right\|_{q,p}:=\left\|\left\{\left\|f\chi_{_{Q_k}}\right\|_{q}\right\}_{k\in\mathbb{Z}^d}\right\|_{\ell^{p}}<\infty,$$ where $Q_k=k+\left[0,1\right)^{d}$ for $k\in\mathbb Z^d$ (see 
\cite{BDD}, \cite{RBHS}, \cite{FSTW}, \cite{FH} and \cite{JSTW} for details). It is advisable to point out that the Wiener amalgam spaces $(L^q,\ell^p)$ are special cases of the Orlicz-slice spaces introduced by Zhang et al. in \cite{ZYYW}, which are themselves special cases of the ball (quasi-)Banach function spaces introduced by Sawano et al. in \cite{SHYY}. Wiener amalgam spaces are also isomorphic to special cases of the mixed-norm spaces defined by Benedek and Panzone \cite{BPR} (see 
\cite{RBHS} and \cite{FSTW} for details). For some properties and a survey of mixed-norm function spaces, the reader can refer to the papers of Hart et al. \cite{HTWX} and, Huang et al. \cite{HYD}.  

As for classical Hardy spaces, not only the Hardy-amalgam spaces can be characterized in terms of grand maximal functions, but also their definition do not depend on the particular function $\varphi$, and the regular function $\varphi$ can be replaced by the Poisson kernel. These spaces admit atomic characterizations with atoms which are exactly those used in classical Hardy spaces, when $0<q\leq1$ and $q\leq p<\infty$ (see \cite{AbFt} and \cite{AbFt2}). Recently in \cite{AbFt3}, we have generalized the characterizations of their dual spaces obtained in \cite{AbFt1} and \cite{AbFt2} for $0<q\leq p\leq 1$ in case $0<q\leq 1$ and $q\leq p<\infty$.
 
We point out that several developments and generalizations of the Hardy spaces theory modeled on the above mentioned generalizations of Wiener amalgam spaces, were obtained by many authors. On the model of ball quasi-Banach function spaces, we can quote the Hardy spaces for ball quasi-Banach spaces introduced by Sawano et al. in \cite{SHYY}, and the Orlicz-slice Hardy spaces of Zhang et al. in \cite{ZYYW}. Also, we have the works of Wang et al. in 
\cite{WYYg} and \cite{WYYZg}, those of Yan et al. in \cite{YYYn}, of Zhang et al. in \cite{ZWYY}, and to complete, the paper of Chang et al. \cite{CWYZg}. On the other hand, on the model of mixed-norm function spaces, we can mention the anisotropic mixed-norm Hardy spaces defined by Cleanthous et al. in \cite{CGNM} and the works of Huang et al. 
\cite{HLYY1}, \cite{HLYY2} and \cite{HLYY3} on these spaces. Several of our results in 
\cite{AbFt}, \cite{AbFt1} and \cite{AbFt2} have been generalized in the context of these general spaces, namely atomic and molecular decompositions, boundedness of Calder\'on-Zygmund operators, convolution and pseudo-differential operators, and many others. Furthermore, characterizations of the dual spaces of these generalized Hardy spaces have been established, covering those of Hardy-amalgam spaces obtained in \cite{AbFt1} and \cite{AbFt2} when $0<q\leq p\leq 1$. However, the characterization of the dual spaces of the Hardy-amalgam spaces $\mathcal H^{(q,p)}$ and $\mathcal{H}_{\mathrm{loc}}^{(q,p)}$ for the exponent range $0<q\leq1<p<\infty$, does not fall within the scope of what has been done in the context of these generalized Hardy spaces. The characterization of the dual spaces of the Hardy-amalgam spaces $\mathcal H^{(q,p)}$ and $\mathcal{H}_{\mathrm{loc}}^{(q,p)}$ for $0<q\leq1<p<\infty$ was so left an open problem until our earlier paper \cite{AbFt3} where an answer was brought. 

The aim of this paper is to two kinds. First, to answer to two questions raised in \cite{AbFt3}; namely whether the inclusion of $\mathcal H^{(1,p)}$ in $(L^1,\ell^p)$ for $1\leq p<\infty$ and the one of $\mathcal H^{(q,p)}$ in $\mathcal{H}_{\mathrm{loc}}^{(q,p)}$ for $0<q\leq1$ and $q\leq p<\infty$, are strict. Next, to generalize some boundedness results of Calder\'on-Zygmund and convolution operators known in the context of the spaces $\mathcal H^1$ and $BMO(\mathbb{R}^d)$ to the case of the Hardy-amalgam spaces $\mathcal H^{(q,p)}$ and their dual spaces when $0<q\leq1$ and $q\leq p<\infty$. To this end, we organize this paper as follows.  

In Section 2, we recall some properties of the Hardy-amalgam spaces $\mathcal H^{(q,p)}$ and $\mathcal{H}_{\mathrm{loc}}^{(q,p)}$, and their dual spaces. Section 3 is devoted to the study of the inclusion of $\mathcal H^{(1,p)}$ in $(L^1,\ell^p)$ for $1\leq p<\infty$ and more generally, the one of $\mathcal H^{(q,p)}$ in $\mathcal{H}_{\mathrm{loc}}^{(q,p)}$ for $0<q\leq1$ and $q\leq p<\infty$. In the last section, we study the boundedness of Calder\'on-Zygmund and convolution operators on the dual spaces of the Hardy-amalgam spaces $\mathcal H^{(q,p)}$ when $0<q\leq1$ and $q\leq p<\infty$.

Throughout the paper, we always let $\mathbb{N}=\left\{1,2,\ldots\right\}$ and $\mathbb{Z}_{+}=\mathbb{N}\cup\left\{0\right\}$. We use $\mathcal S := \mathcal S(\mathbb R^{d})$ to denote the Schwartz class of rapidly decreasing smooth functions equipped with the topology defined by the family of norms $\left\{\mathcal{N}_{m}\right\}_{m\in\mathbb{Z}_{+}}$, where for all $m\in\mathbb{Z}_{+}$ and $\psi\in\mathcal{S}$, $$\mathcal{N}_{m}(\psi):=\underset{x\in\mathbb R^{d}}\sup(1 + |x|)^{m}\underset{|\beta|\leq m}\sum|{\partial}^\beta \psi(x)|,$$ with $|\beta|=\beta_1+\ldots+\beta_d$, ${\partial}^\beta=\left(\partial/{\partial x_1}\right)^{\beta_1}\ldots\left(\partial/{\partial x_d}\right)^{\beta_d}$ for all $\beta=(\beta_1,\ldots,\beta_d)\in\mathbb{Z}_{+}^d$ and $|x|:=(x_1^2+\ldots+x_d^2)^{1/2}$. The dual space of $\mathcal S$ is the space of tempered distributions denoted by $\mathcal S':= \mathcal S'(\mathbb R^{d})$ equipped with the weak-$\ast$ topology. If $f\in\mathcal{S'}$ and $\theta\in\mathcal{S}$, we denote the evaluation of $f$ on $\theta$ by $\left\langle f,\theta\right\rangle$. The letter $C$ will be used for non-negative constants independent of the relevant variables that may change from one occurrence to another. When a constant depends on some important parameters $\alpha,\gamma,\ldots$, we denote it by $C(\alpha,\gamma,\ldots)$. Constants with subscript, such as $C_{\alpha,\gamma,\ldots}$, do not change in different occurrences but depend on the parameters mentioned in them. We propose the following abbreviation $\mathrm{\bf A}\lsim \mathrm{\bf B}$ for the inequalities $\mathrm{\bf A}\leq C\mathrm{\bf B}$, where $C$ is a positive constant independent of the main parameters. If $\mathrm{\bf A}\lsim \mathrm{\bf B}$ and $\mathrm{\bf B}\lsim \mathrm{\bf A}$, then we write $\mathrm{\bf A}\approx \mathrm{\bf B}$. For any given quasi-normed spaces $\mathcal{A}$ and $\mathcal{B}$ with the corresponding quasi-norms $\left\|\cdot\right\|_{\mathcal{A}}$ and $\left\|\cdot\right\|_{\mathcal{B}}$, the notation $\mathcal{A}\hookrightarrow\mathcal{B}$ means that if $f\in\mathcal{A}$, then $f\in\mathcal{B}$ and $\left\|f\right\|_{\mathcal{B}}\lsim\left\|f\right\|_{\mathcal{A}}$. Also, $\mathcal{A}\cong\mathcal{B}$ means that $\mathcal{A}$ is isomorphic to $\mathcal{B}$, with equivalence of the quasi-norms $\left\|\cdot\right\|_{\mathcal{A}}$ and $\left\|\cdot\right\|_{\mathcal{B}}$.

For a real number $\lambda>0$ and a cube $Q\subset\mathbb R^{d}$ (by a cube we mean a cube whose edges are parallel to the coordinate axes), we write $\lambda Q$ for the cube with same center as $Q$ and side-length $\lambda$ times side-length of $Q$, while $\left\lfloor \lambda \right\rfloor$ stands for the greatest integer less or equal to $\lambda$. Also, for $x\in\mathbb R^{d}$ and $\ell>0$, $Q(x,\ell)$ will denote the cube centered at $x$ and side-length $\ell$. We use the same notations for balls. For a measurable set $E\subset\mathbb R^d$, we denote by $\chi_{_{E}}$ the characteristic function of $E$ and by $\left|E\right|$  its Lebesgue measure. To finish, we denote by $\mathcal{Q}$ the set of all cubes of $\mathbb R^{d}$.

\section{Prerequisites on Hardy-amalgam spaces and their dual spaces}

\subsection{On Hardy-amalgam spaces} 

Let $0<q,p<\infty$. The Hardy-amalgam spaces $\mathcal{H}^{(q,p)}$ and $\mathcal{H}_{\mathrm{loc}}^{(q,p)}$ are Banach spaces whenever $1\leq q,p<\infty$ and quasi-Banach spaces otherwise (see \cite[Proposition 3.8]{AbFt}). Moreover, for $0<q\leq 1$ and $q\leq p<\infty$, they admit atomic characterizations with atoms which are exactly those used in classical Hardy spaces (see \cite[Theorems 4.3, 4.4 and 4.6]{AbFt} and \cite[Theorems 3.2, 3.3 and 3.9]{AbFt2}). 

We recall that for $0<q\leq 1$, $q\leq p<\infty$, $1<r\leq\infty$ and $s\geq\left\lfloor d\left(\frac{1}{q}-1\right)\right\rfloor$ being an integer, a function $\textbf{a}$ is a $(q,r,s)$-atom on $\mathbb{R}^d$ for $\mathcal{H}^{(q,p)}$ if there exists a cube $Q$ such that  
\begin{enumerate}
\item $\text{supp}(\textbf{a})\subset Q$;
\item $\left\|\textbf{a}\right\|_r\leq|Q|^{\frac{1}{r}-\frac{1}{q}}$; \label{defratom1}
\item $\int_{\mathbb{R}^d}x^{\beta}\textbf{a}(x)dx=0$, for all multi-indexes $\beta$ with $|\beta|\leq s$. \label{defratom2}
\end{enumerate}

We define in the same way the atoms of the local Hardy-amalgam spaces, namely the local $(q,r,\delta)$-atoms. But, just like those of the classical local Hardy spaces, for these local $(q,r,\delta)$-atoms, only Condition \ref{defratom2}. is not requiered when the corresponding cubes $Q$ have side length greater than or equal to 1. We denote by $\mathcal{A}(q,r,s)$ the set of all $(\textbf{a},Q)$ such that $\textbf{a}$ is a $(q,r,s)$-atom and $Q$ is the associated cube, and $\mathcal{A}_{\mathrm{loc}}(q,r,s)$ for the local $(q,r,\delta)$-atoms.  

We suppose that $0<q\leq 1$ and $q\leq p<\infty$. Let $1<r\leq\infty$ and $\delta\geq\left\lfloor d\left(\frac{1}{q}-1\right)\right\rfloor$ be an integer. We denote by $\mathcal{H}_{fin}^{(q,p)}$ the subspace of $\mathcal{H}^{(q,p)}$ consisting of finite linear combinations of $(q,r,\delta)$-atoms, and by $\mathcal{H}_{\mathrm{loc},fin}^{(q,p)}$ the subspace of $\mathcal{H}_{\mathrm{loc}}^{(q,p)}$ consisting of finite linear combinations of local $(q,r,\delta)$-atoms. The spaces $\mathcal{H}_{fin}^{(q,p)}$ and $\mathcal{H}_{\mathrm{loc},fin}^{(q,p)}$ are respectively dense subspaces of $\mathcal{H}^{(q,p)}$ and $\mathcal{H}_{\mathrm{loc}}^{(q,p)}$ (see  \cite[Remark 4.7]{AbFt} and \cite[Remark 3.12]{AbFt2}).

\subsection{On the dual space of $\mathcal{H}^{(q,p)}$}

Let's fix an integer $\delta\geq0$. Let $1\leq r<\infty$, $g\in L_{\mathrm{loc}}^r$ and $\Omega\subsetneq\mathbb R^d$  an open subset. We define $$\textit{O}(g,\Omega,r):=\sup{\sum_{n\geq0}|\widetilde{Q^n}|^{\frac{1}{r'}}\left(\int_{\widetilde{Q^n}}\left|g(x)-P_{\widetilde{Q^n}}^{\delta}(g)(x)\right|^{r}dx\right)^{\frac{1}{r}}},$$ where $\frac{1}{r}+\frac{1}{r'}=1$ and the supremum is taken over all families of cubes $\left\{Q^n\right\}_{n\geq0}$ such that $Q^n\subset\Omega$ for all $n\geq0$ and $\sum_{n\geq0}\chi_{_{Q^n}}\leq K(d)$, with $\widetilde{Q^n}=C_0Q^n$, where $K(d)>1$ and $C_0>1$ are fixed constants independent of $\Omega$ and $\left\{Q^n\right\}_{n\geq0}$, and for a cube $Q$, $P_Q^{\delta}(g)$ stands for the unique  polynomial of $\mathcal{P_{\delta}}$ ($\mathcal{P_{\delta}}:=\mathcal{P_{\delta}}(\mathbb{R}^d)$ is the space of polynomial functions of degree at most $\delta$)  such that, for all $\mathfrak{q}\in\mathcal{P_{\delta}}$,
\begin{align*}
\int_{Q}\left[g(x)-P_Q^{\delta}(g)(x)\right]\mathfrak{q}(x)dx=0. 
\end{align*}  
We consider the functions $\phi_{1},\ \phi_{2}\ \text{and}\ \phi_{3}:\mathcal{Q}\rightarrow(0,\infty)$ defined by 
\begin{align}
\phi_{1}(Q)=\frac{\left\|\chi_{_{Q}}\right\|_{q,p}}{|Q|}\ ,\ \phi_{2}(Q)=\frac{\left\|\chi_{_{Q}}\right\|_q}{|Q|}\ \text{ and }\ \phi_{3}(Q)=\frac{\left\|\chi_{_{Q}}\right\|_p}{|Q|}\ , 
\label{Campanatonolocnloc}
\end{align} 
for all $Q\in\mathcal{Q}$, whenever $0<q\leq 1$ and $0<p<\infty$.
\begin{defn}\label{corprecis}
Suppose that $0<q\leq1$ and $0<p<\infty$. Let $0<\eta<\infty$ and $1\leq r<\infty$. We say that a function $g$ in $L_{\mathrm{loc}}^r$ belongs to $\mathcal{L}_{r,\phi_{1},\delta}^{(q,p,\eta)}:=\mathcal{L}_{r,\phi_{1},\delta}^{(q,p,\eta)}(\mathbb{R}^d)$ if there is a constant $C>0$ such that, for all families of open subsets $\left\{\Omega^j\right\}_{j\in\mathbb{Z}}$ with $\left\|\sum_{j\in\mathbb{Z}}2^{j\eta}\chi_{\Omega^j}\right\|_{\frac{q}{\eta},\frac{p}{\eta}}<\infty$, we have 
\begin{align}
\sum_{j\in\mathbb{Z}}2^j\textit{O}(g,\Omega^j,r)\leq C\left\|\sum_{j\in\mathbb{Z}}2^{j\eta}\chi_{\Omega^j}\right\|_{\frac{q}{\eta},\frac{p}{\eta}}^{\frac{1}{\eta}}. \label{dualqp}
\end{align} 
\end{defn}

We have $\mathcal{P_{\delta}}\subset\mathcal{L}_{r,\phi_{1},\delta}^{(q,p,\eta)}$. When $g\in\mathcal{L}_{r,\phi_{1},\delta}^{(q,p,\eta)}$, we put $$\left\|g\right\|_{\mathcal{L}_{r,\phi_{1},\delta}^{(q,p,\eta)}}:=\inf\left\{C>0:\ C \text{ satisfies } (\ref{dualqp})\right\}.$$ Then $\left\|\cdot\right\|_{\mathcal{L}_{r,\phi_{1},\delta}^{(q,p,\eta)}}$ defines a semi-norm on $\mathcal{L}_{r,\phi_{1},\delta}^{(q,p,\eta)}$ and a norm on $\mathcal{L}_{r,\phi_{1},\delta}^{(q,p,\eta)}/\mathcal{P_{\delta}}$. In the sequel, $\mathcal{L}_{r,\phi_{1},\delta}^{(q,p,\eta)}$ will designate $\mathcal{L}_{r,\phi_{1},\delta}^{(q,p,\eta)}/\mathcal{P_{\delta}}$. 
 
\begin{defn}\cite[Definition 6.1]{NEYS} Let $1\leq r\leq\infty$, a function $\phi: \mathcal{Q}\rightarrow (0,\infty)$  and $f\in L_{\mathrm{loc}}^r$. One denotes 
\begin{align*}
\left\|f\right\|_{\mathcal{L}_{r,\phi,\delta}}:=\sup_{Q\in\mathcal{Q}}\frac{1}{\phi(Q)}\left(\frac{1}{|Q|}\int_{Q}\left|f(x)-P_Q^{\delta}(f)(x)\right|^rdx\right)^{\frac{1}{r}}, 
\end{align*}
when $r<\infty$, and
\begin{align*}
\left\|f\right\|_{\mathcal{L}_{r,\phi,\delta}}:=\sup_{Q\in\mathcal{Q}}\frac{1}{\phi(Q)}\left\|f-P_Q^{\delta}(f)\right\|_{L^{\infty}(Q)}, 
\end{align*}
when $r=\infty$. Then, the Campanato space $\mathcal{L}_{r,\phi,\delta}(\mathbb{R}^d)$ is defined to be the set of all $f\in L_{\mathrm{loc}}^r$ such that $\left\|f\right\|_{\mathcal{L}_{r,\phi,\delta}}<\infty$. One considers elements in $\mathcal{L}_{r,\phi,\delta}(\mathbb{R}^d)$ modulo polynomials of degree $\delta$ so that $\mathcal{L}_{r,\phi,\delta}(\mathbb{R}^d)$ is a Banach space. When one writes $f\in\mathcal{L}_{r,\phi,\delta}(\mathbb{R}^d)$, then $f$ stands for the representative of $\left\{f+\mathfrak{q}: \mathfrak{q}\ \text{is a polynomial of degree}\ \delta\right\}$.  
\end{defn}    

The following inequalities were proved in \cite{AbFt3}.\\ 

For $0<q\leq1$, $0<p<\infty$, $0<\eta<\infty$ and $1\leq r<\infty$, we have  
\begin{align}
\left\|g\right\|_{\mathcal{L}_{r,\phi_{1},\delta}}\leq\left\|g\right\|_{\mathcal{L}_{r,\phi_{1},\delta}^{(q,p,\eta)}}, \label{dualqp3}
\end{align}
for all $g\in\mathcal{L}_{r,\phi_{1},\delta}^{(q,p,\eta)}$, and hence  
\begin{align}
\left\|g\right\|_{\mathcal{L}_{r,\phi_{2},\delta}}\leq\left\|g\right\|_{\mathcal{L}_{r,\phi_{1},\delta}}\leq\left\|g\right\|_{\mathcal{L}_{r,\phi_{1},\delta}^{(q,p,\eta)}}, \label{dualqp3bis}
\end{align}
for all $g\in\mathcal{L}_{r,\phi_{1},\delta}^{(q,p,\eta)}$, whenever $0<q\leq1$ and $q\leq p<\infty$. Moreover, whenever $0<q,p\leq1$ and $0<\eta\leq1$, we have 
\begin{align}
\left\|g\right\|_{\mathcal{L}_{r,\phi_{1},\delta}^{(q,p,\eta)}}\approx\left\|g\right\|_{\mathcal{L}_{r,\phi_{1},\delta}}, \label{dualqp3bisbis}
\end{align}
which means that $\mathcal{L}_{r,\phi_{1},\delta}^{(q,p,\eta)}=\mathcal{L}_{r,\phi_{1},\delta}$  with equivalent norms. 
 
Also, for $0<q\leq1$, $0<p<\infty$, $0<\eta<\infty$ and $1\leq r<\infty$, the space $\mathcal{L}_{r,\phi_{1},\delta}^{(q,p,\eta)}$ endowed with the norm $\left\|\cdot\right\|_{\mathcal{L}_{r,\phi_{1},\delta}^{(q,p,\eta)}}$ is complete (see \cite[Proposition 3.5]{AbFt3}).

For $T\in\left(\mathcal{H}^{(q,p)}\right)^{\ast}$ the topological dual space of $\mathcal{H}^{(q,p)}$, we put $$\left\|T\right\|:=\left\|T\right\|_{\left(\mathcal{H}^{(q,p)}\right)^{\ast}}=\sup_{\underset{\left\|f\right\|_{\mathcal{H}^{(q,p)}}\leq 1}{f\in\mathcal{H}^{(q,p)}}}|T(f)|.$$ The hereafter results were proved in \cite{AbFt3}.   

\begin{thm}\cite[Theorem 3.7]{AbFt3}\label{theoremdualqp}
Suppose that $0<q\leq1<p<\infty$ and $\delta\geq\left\lfloor d\left(\frac{1}{q}-1\right)\right\rfloor$. Let $p<r\leq\infty$. Then the topological dual space $\left(\mathcal{H}^{(q,p)}\right)^{\ast}$ of the Hardy-amalgam space $\mathcal H^{(q,p)}$ is isomorphic to $\mathcal{L}_{r',\phi_1,\delta}^{(q,p,\eta)}$  with equivalent norms, where $\frac{1}{r}+\frac{1}{r'}=1$ and, $0<\eta<q$ if $r<\infty$ and $0<\eta\leq1$ if not. More precisely, we have the following assertions:
\begin{enumerate}
\item Let $g\in\mathcal{L}_{r',\phi_1,\delta}^{(q,p,\eta)}$ and $\mathcal{H}_{fin}^{(q,p)}$ be the subspace of $\mathcal{H}^{(q,p)}$ consisting of finite linear combinations of $(q,r,\delta)$-atoms. Then the mapping $$T_g:\mathcal{H}_{fin}^{(q,p)}\ni f\longmapsto\int_{\mathbb{R}^d}g(x)f(x)dx,$$ extends to a unique continuous linear functional $\widetilde{T_g}$ on $\mathcal{H}^{(q,p)}$ such that 
\begin{eqnarray*}
\left\|\widetilde{T_g}\right\|=\left\|T_g\right\|\leq C\left\|g\right\|_{\mathcal{L}_{r',\phi_{1},\delta}^{(q,p,\eta)}}, 
\end{eqnarray*} 
where $C>0$ is a constant independent of $g$. \label{dualpointqp1}
\item Conversely, for every $T\in\left(\mathcal{H}^{(q,p)}\right)^{\ast}$, there exists $g\in\mathcal{L}_{r',\phi_1,\delta}^{(q,p,\eta)}$ such that $T=T_g$; namely $$T(f)=\int_{\mathbb{R}^d}g(x)f(x)dx,\ \text{ for all }\ f\in\mathcal{H}_{fin}^{(q,p)},$$ and
\begin{eqnarray*}
\left\|g\right\|_{\mathcal{L}_{r',\phi_{1},\delta}^{(q,p,\eta)}}\leq C\left\|T\right\|, 
\end{eqnarray*} 
where $C>0$ is a constant independent of $T$. \label{dualpointqp2}
\end{enumerate}
\end{thm}

\begin{remark}\cite[Remark 3.1]{AbFt3}\label{remarqedualeqp1}
Suppose that $0<q\leq1<p<\infty$ and $\delta\geq\left\lfloor d\left(\frac{1}{q}-1\right)\right\rfloor$. Let $1<r<p'$, $0<\eta<q$ and $0<\eta_1\leq1$. Then we have $$\mathcal{L}_{1,\phi_1,\delta}^{(q,p,\eta_1)}\cong\left(\mathcal{H}^{(q,p)}\right)^{\ast}\cong\mathcal{L}_{r,\phi_1,\delta}^{(q,p,\eta)}\hookrightarrow\mathcal{L}_{r,\phi_2,\delta}\cong\left(\mathcal{H}^q\right)^{\ast}$$ and $$\mathcal{L}_{1,\phi_1,\delta}^{(q,p,\eta_1)}\cong\left(\mathcal{H}^{(q,p)}\right)^{\ast}\cong\mathcal{L}_{r,\phi_1,\delta}^{(q,p,\eta)}\hookrightarrow L^{p'}\cong\left(\mathcal{H}^p\right)^{\ast}.$$ For $q=1$, we have $$(L^{\infty},\ell^{p'})\hookrightarrow\mathcal{L}_{1,\phi_1,\delta}^{(1,p,\eta_1)}\cong\mathcal{L}_{r,\phi_1,\delta}^{(1,p,\eta)}\hookrightarrow BMO(\mathbb{R}^d)$$ and $$(L^{\infty},\ell^{p'})\hookrightarrow\mathcal{L}_{1,\phi_1,\delta}^{(1,p,\eta_1)}\cong\mathcal{L}_{r,\phi_1,\delta}^{(1,p,\eta)}\hookrightarrow L^{p'}.$$  
 Moreover, the inclusion of $(L^{\infty},\ell^{p'})$ in $\mathcal{L}_{1,\phi_1,\delta}^{(1,p,\eta_1)}\cong\mathcal{L}_{r,\phi_1,\delta}^{(1,p,\eta)}$ is strict.  
\end{remark} 

\begin{thm}\cite[Theorem 3.8]{AbFt3}\label{theoremdualunifie}
Suppose that $0<q\leq1$, $q\leq p<\infty$ and $\delta\geq\left\lfloor d\left(\frac{1}{q}-1\right)\right\rfloor$. Let $\max\left\{1,p\right\}<r\leq\infty$. Then $\left(\mathcal{H}^{(q,p)}\right)^{\ast}$ is isomorphic to $\mathcal{L}_{r',\phi_1,\delta}^{(q,p,\eta)}$, where $\frac{1}{r}+\frac{1}{r'}=1$, $0<\eta<q$ if $r<\infty$, and $0<\eta\leq1$ if $r=\infty$, with equivalent norms. More precisely, we have the following assertions:
\begin{enumerate}
\item Let $g\in\mathcal{L}_{r',\phi_1,\delta}^{(q,p,\eta)}$ and $\mathcal{H}_{fin}^{(q,p)}$ be the subspace of $\mathcal{H}^{(q,p)}$ consisting of finite linear combinations of $(q,r,\delta)$-atoms. Then the mapping $$T_g:\mathcal{H}_{fin}^{(q,p)}\ni f\longmapsto\int_{\mathbb{R}^d}g(x)f(x)dx,$$ extends to a unique continuous linear functional $\widetilde{T_g}$ on $\mathcal{H}^{(q,p)}$ such that 
\begin{eqnarray*}
\left\|\widetilde{T_g}\right\|=\left\|T_g\right\|\leq C\left\|g\right\|_{\mathcal{L}_{r',\phi_{1},\delta}^{(q,p,\eta)}}, 
\end{eqnarray*} 
where $C>0$ is a constant independent of $g$. 
\item Conversely, for any $T\in\left(\mathcal{H}^{(q,p)}\right)^{\ast}$, there exists $g\in\mathcal{L}_{r',\phi_1,\delta}^{(q,p,\eta)}$ such that $T=T_g$; namely $$T(f)=\int_{\mathbb{R}^d}g(x)f(x)dx,\ \text{ for all }\ f\in\mathcal{H}_{fin}^{(q,p)},$$ and
\begin{eqnarray*}
\left\|g\right\|_{\mathcal{L}_{r',\phi_{1},\delta}^{(q,p,\eta)}}\leq C\left\|T\right\|, 
\end{eqnarray*} 
where $C>0$ is a constant independent of $T$. 
\end{enumerate}
\end{thm}

In Theorem \ref{theoremdualunifie}, when $p\leq1$, we can take $0<\eta\leq1$ for $1<r\leq\infty$.

\subsection{On the dual space of $\mathcal{H}_{\mathrm{loc}}^{(q,p)}$}

Let's fix an integer $\delta\geq0$. Let $1\leq r<\infty$, $g$ be a function in $L_{\mathrm{loc}}^r$ and $\Omega$ be an open subset such that $\Omega\neq\mathbb{R}^d$. We put
\begin{align*}\textit{O}(g,\Omega,r)^{\mathrm{loc}}&:=\sup\left[\sum_{|\widetilde{Q^n}|<1}|\widetilde{Q^n}|^{\frac{1}{r'}}\left(\int_{\widetilde{Q^n}}\left|g(x)-P_{\widetilde{Q^n}}^{\delta}(g)(x)\right|^{r}dx\right)^{\frac{1}{r}}\right. \\
 &+\left. \sum_{|\widetilde{Q^n}|\geq1}|\widetilde{Q^n}|^{\frac{1}{r'}}\left(\int_{\widetilde{Q^n}}|g(x)|^{r}dx\right)^{\frac{1}{r}}\right],
\end{align*}
where $\frac{1}{r}+\frac{1}{r'}=1$ and the supremum is taken over all families of cubes $\left\{Q^n\right\}_{n\geq0}$ such that $Q^n\subset\Omega$ for all $n\geq0$ and $\sum_{n\geq0}\chi_{_{Q^n}}\leq K(d)$, with $\widetilde{Q^n}=C_0Q^n$, $K(d)>1$ and $C_0>1$ are the same constants as in the definition of $\textit{O}(g,\Omega,r)$.   

\begin{defn}
Suppose that $0<q\leq1$ and $0<p<\infty$. Let $0<\eta<\infty$ and $1\leq r<\infty$. We say that a function $g$ in $L_{\mathrm{loc}}^r$ belongs to $\mathcal{L}_{r,\phi_{1},\delta}^{(q,p,\eta)\mathrm{loc}}:=\mathcal{L}_{r,\phi_{1},\delta}^{(q,p,\eta)\mathrm{loc}}(\mathbb{R}^d)$ if there exists a constant $C>0$ such that, for all families of open subsets $\left\{\Omega^j\right\}_{j\in\mathbb{Z}}$ with $\left\|\sum_{j\in\mathbb{Z}}2^{j\eta}\chi_{\Omega^j}\right\|_{\frac{q}{\eta},\frac{p}{\eta}}<\infty$, we have
\begin{align}
\sum_{j\in\mathbb{Z}}2^j\textit{O}(g,\Omega^j,r)^{\mathrm{loc}}\leq C\left\|\sum_{j\in\mathbb{Z}}2^{j\eta}\chi_{\Omega^j}\right\|_{\frac{q}{\eta},\frac{p}{\eta}}^{\frac{1}{\eta}}. \label{dualqploc}
\end{align} 
\end{defn}
 
We define $\left\|g\right\|_{\mathcal{L}_{r,\phi_{1},\delta}^{(q,p,\eta)\mathrm{loc}}}:=\inf\left\{C>0:\ C \text{ satisfies } (\ref{dualqploc})\right\}$, when $g\in\mathcal{L}_{r,\phi_{1},\delta}^{(q,p,\eta)\mathrm{loc}}$. $\left\|\cdot\right\|_{\mathcal{L}_{r,\phi_{1},\delta}^{(q,p,\eta)\mathrm{loc}}}$ defines a norm on $\mathcal{L}_{r,\phi_{1},\delta}^{(q,p,\eta)\mathrm{loc}}$.
 
\begin{defn}\cite[Definition 4.1]{AbFt2} Let $1\leq r\leq\infty$ and $\phi: \mathcal{Q}\rightarrow (0,\infty)$ be a function. The space $\mathcal{L}_{r,\phi,\delta}^{\mathrm{loc}}:=\mathcal{L}_{r,\phi,\delta}^{\mathrm{loc}}(\mathbb{R}^d)$ is the set of all $f\in L_{\mathrm{loc}}^r$ such that $\left\|f\right\|_{{\mathcal{L}}_{r,\phi,\delta}^{\mathrm{loc}}}<\infty$, where 
\begin{align*}
\left\|f\right\|_{{\mathcal{L}}_{r,\phi,\delta}^{\mathrm{loc}}}&:=\sup_{\underset{|Q|\geq1}{Q\in\mathcal{Q}}}\frac{1}{\phi(Q)}\left(\frac{1}{|Q|}\int_{Q}|f(x)|^rdx\right)^{\frac{1}{r}}\\
&+\sup_{\underset{|Q|<1}{Q\in\mathcal{Q}}}\frac{1}{\phi(Q)}\left(\frac{1}{|Q|}\int_{Q}\left|f(x)-P_Q^{\delta}(f)(x)\right|^rdx\right)^{\frac{1}{r}},
\end{align*}
when $r<\infty$, and
\begin{align*}
\left\|f\right\|_{\mathcal{L}_{r,\phi,\delta}^{\mathrm{loc}}}:=\sup_{\underset{|Q|\geq1}{Q\in\mathcal{Q}}}\frac{1}{\phi(Q)}\left\|f\right\|_{L^{\infty}(Q)}+\sup_{\underset{|Q|<1}{Q\in\mathcal{Q}}}\frac{1}{\phi(Q)}\left\|f-P_Q^{\delta}(f)\right\|_{L^{\infty}(Q)}, 
\end{align*}
when $r=\infty$.
\end{defn}  

The following inequalities were obtained in \cite{AbFt3}.\\

For $0<q\leq1$, $0<p<\infty$, $0<\eta<\infty$ and $1\leq r<\infty$, we have   
\begin{align}
\left\|g\right\|_{\mathcal{L}_{r,\phi_{1},\delta}^{\mathrm{loc}}}\leq2\left\|g\right\|_{\mathcal{L}_{r,\phi_{1},\delta}^{(q,p,\eta)\mathrm{loc}}}, \label{dualqploc3}
\end{align}
for all $g\in\mathcal{L}_{r,\phi_{1},\delta}^{(q,p,\eta)\mathrm{loc}}$, and hence
\begin{align}
\left\|g\right\|_{\mathcal{L}_{r,\phi_{2},\delta}^{\mathrm{loc}}}\leq\left\|g\right\|_{\mathcal{L}_{r,\phi_{1},\delta}^{\mathrm{loc}}}\leq2\left\|g\right\|_{\mathcal{L}_{r,\phi_{1},\delta}^{(q,p,\eta)\mathrm{loc}}}, \label{dualqploc3bbi}
\end{align}
when $0<q\leq1$ and $q\leq p<\infty$. Futhermore, when $0<q,p\leq1$ and $0<\eta\leq1$, we have  
\begin{align}
\left\|g\right\|_{\mathcal{L}_{r,\phi_{1},\delta}^{(q,p,\eta)\mathrm{loc}}}\approx\left\|g\right\|_{\mathcal{L}_{r,\phi_{1},\delta}^{\mathrm{loc}}}, \label{dualqploc3bisbis}
\end{align} 
which means that $\mathcal{L}_{r,\phi_{1},\delta}^{(q,p,\eta)\mathrm{loc}}=\mathcal{L}_{r,\phi_{1},\delta}^{\mathrm{loc}}$ with equivalent norms. 

We have also for $0<q\leq1$, $0<p<\infty$, $0<\eta<\infty$ and $1\leq r<\infty$, 
\begin{align}
\left\|g\right\|_{\mathcal{L}_{r,\phi_{1},\delta}^{(q,p,\eta)}}\lsim\left\|g\right\|_{\mathcal{L}_{r,\phi_{1},\delta}^{(q,p,\eta)\mathrm{loc}}}, \label{dualqploc3bisbisbis}
\end{align}
for all $g\in\mathcal{L}_{r,\phi_{1},\delta}^{(q,p,\eta)\mathrm{loc}}$. 
 
For $T\in\left(\mathcal{H}_{\mathrm{loc}}^{(q,p)}\right)^{\ast}$ the topological dual space of $\mathcal{H}_{\mathrm{loc}}^{(q,p)}$, we set $$\left\|T\right\|:=\left\|T\right\|_{\left(\mathcal{H}_{\mathrm{loc}}^{(q,p)}\right)^{\ast}}=\sup_{\underset{\left\|f\right\|_{\mathcal{H}_{\mathrm{loc}}^{(q,p)}}\leq 1}{f\in\mathcal{H}_{\mathrm{loc}}^{(q,p)}}}|T(f)|.$$

The hereafter results were proved in \cite{AbFt3}. 

\begin{thm}\cite[Theorem 4.3]{AbFt3} \label{theoremdualqploc}
Suppose that $0<q\leq1<p<\infty$ and $\delta\geq\left\lfloor d\left(\frac{1}{q}-1\right)\right\rfloor$. Let $p<r\leq\infty$. Then $\left(\mathcal{H}_{\mathrm{loc}}^{(q,p)}\right)^{\ast}$ is isomorphic to $\mathcal{L}_{r',\phi_1,\delta}^{(q,p,\eta)\mathrm{loc}}$, where $\frac{1}{r}+\frac{1}{r'}=1$, $0<\eta<q$ if $r<\infty$, and $0<\eta\leq1$ if $r=\infty$, with equivalent norms. More precisely, we have the following assertions:  
\begin{enumerate}
\item Let $g\in\mathcal{L}_{r',\phi_1,\delta}^{(q,p,\eta)\mathrm{loc}}$ and $\mathcal{H}_{\mathrm{loc},fin}^{(q,p)}$ be the subspace of $\mathcal{H}_{\mathrm{loc}}^{(q,p)}$ consisting of finite linear combinations of local $(q,r,\delta)$-atoms. Then the mapping $$T_g:\mathcal{H}_{\mathrm{loc},fin}^{(q,p)}\ni f\longmapsto\int_{\mathbb{R}^d}g(x)f(x)dx,$$ extends to a unique continuous linear functional $\widetilde{T_g}$ on $\mathcal{H}_{\mathrm{loc}}^{(q,p)}$ such that  
\begin{eqnarray*}
\left\|\widetilde{T_g}\right\|=\left\|T_g\right\|\leq C\left\|g\right\|_{\mathcal{L}_{r',\phi_{1},\delta}^{(q,p,\eta)\mathrm{loc}}}, 
\end{eqnarray*} 
where $C>0$ is a constant independent of $g$. \label{dualpointqploc1}
\item Conversely, for any $T\in\left(\mathcal{H}_{\mathrm{loc}}^{(q,p)}\right)^{\ast}$, there exists $g\in\mathcal{L}_{r',\phi_1,\delta}^{(q,p,\eta)\mathrm{loc}}$ such that $T=T_g$; namely $$T(f)=\int_{\mathbb{R}^d}g(x)f(x)dx,\ \text{ for all }\ f\in\mathcal{H}_{\mathrm{loc},fin}^{(q,p)},$$ and
\begin{eqnarray*}
\left\|g\right\|_{\mathcal{L}_{r',\phi_{1},\delta}^{(q,p,\eta)\mathrm{loc}}}\leq C\left\|T\right\|, 
\end{eqnarray*} 
where $C>0$ is a constant independent of $T$. \label{dualpointqploc2}
\end{enumerate}
\end{thm}

\begin{remark}\cite[Remark 4.1]{AbFt3}\label{remarqedualeqploc1}
Suppose that $0<q\leq1<p<\infty$ and $\delta\geq\left\lfloor d\left(\frac{1}{q}-1\right)\right\rfloor$. Let $1<r<p'$, $0<\eta<q$ and $0<\eta_1\leq1$. We have $$\mathcal{L}_{1,\phi_1,\delta}^{(q,p,\eta_1)\mathrm{loc}}\cong\left(\mathcal{H}_{\mathrm{loc}}^{(q,p)}\right)^{\ast}\cong\mathcal{L}_{r,\phi_1,\delta}^{(q,p,\eta)\mathrm{loc}}\hookrightarrow\mathcal{L}_{r,\phi_2,\delta}^{\mathrm{loc}}\cong\left(\mathcal{H}_{\mathrm{loc}}^q\right)^{\ast}$$ and $$\mathcal{L}_{1,\phi_1,\delta}^{(q,p,\eta_1)\mathrm{loc}}\cong\left(\mathcal{H}_{\mathrm{loc}}^{(q,p)}\right)^{\ast}\cong\mathcal{L}_{r,\phi_1,\delta}^{(q,p,\eta)\mathrm{loc}}\hookrightarrow L^{p'}\cong\left(\mathcal{H}_{\mathrm{loc}}^p\right)^{\ast}.$$ In particular when $q=1$, we have $$(L^{\infty},\ell^{p'})\hookrightarrow\mathcal{L}_{1,\phi_1,\delta}^{(1,p,\eta_1)\mathrm{loc}}\cong\mathcal{L}_{r,\phi_1,\delta}^{(1,p,\eta)\mathrm{loc}}\hookrightarrow bmo(\mathbb{R}^d)$$ and $$(L^{\infty},\ell^{p'})\hookrightarrow\mathcal{L}_{1,\phi_1,\delta}^{(1,p,\eta_1)\mathrm{loc}}\cong\mathcal{L}_{r,\phi_1,\delta}^{(1,p,\eta)\mathrm{loc}}\hookrightarrow L^{p'}.$$ 
\end{remark}  

\begin{thm}\cite[Theorem 4.4]{AbFt3} \label{theoremdualunifieloc}
Suppose that $0<q\leq1$, $q\leq p<\infty$ and $\delta\geq\left\lfloor d\left(\frac{1}{q}-1\right)\right\rfloor$. Let $\max\left\{1,p\right\}<r\leq\infty$. Then $\left(\mathcal{H}_{\mathrm{loc}}^{(q,p)}\right)^{\ast}$ is isomorphic to $\mathcal{L}_{r',\phi_1,\delta}^{(q,p,\eta)\mathrm{loc}}$, where $\frac{1}{r}+\frac{1}{r'}=1$, $0<\eta<q$ if $r<\infty$, and $0<\eta\leq1$ if $r=\infty$, with equivalent norms. More precisely, we have the following assertions:  
\begin{enumerate}
\item Let $g\in\mathcal{L}_{r',\phi_1,\delta}^{(q,p,\eta)\mathrm{loc}}$ and $\mathcal{H}_{\mathrm{loc},fin}^{(q,p)}$ be the subspace of $\mathcal{H}_{\mathrm{loc}}^{(q,p)}$ consisting of finite linear combinations of local $(q,r,\delta)$-atoms. Then the mapping $$T_g:\mathcal{H}_{\mathrm{loc},fin}^{(q,p)}\ni f\longmapsto\int_{\mathbb{R}^d}g(x)f(x)dx,$$ extends to a unique continuous linear functional $\widetilde{T_g}$ on $\mathcal{H}_{\mathrm{loc}}^{(q,p)}$ such that  
\begin{eqnarray*}
\left\|\widetilde{T_g}\right\|=\left\|T_g\right\|\leq C\left\|g\right\|_{\mathcal{L}_{r',\phi_{1},\delta}^{(q,p,\eta)\mathrm{loc}}}, 
\end{eqnarray*} 
where $C>0$ is a constant independent of $g$. 
\item Conversely, for any $T\in\left(\mathcal{H}_{\mathrm{loc}}^{(q,p)}\right)^{\ast}$, there exists $g\in\mathcal{L}_{r',\phi_1,\delta}^{(q,p,\eta)\mathrm{loc}}$ such that $T=T_g$; namely $$T(f)=\int_{\mathbb{R}^d}g(x)f(x)dx,\ \text{ for all }\ f\in\mathcal{H}_{\mathrm{loc},fin}^{(q,p)},$$ and
\begin{eqnarray*}
\left\|g\right\|_{\mathcal{L}_{r',\phi_{1},\delta}^{(q,p,\eta)\mathrm{loc}}}\leq C\left\|T\right\|, 
\end{eqnarray*} 
where $C>0$ is a constant independent of $T$. 
\end{enumerate}
\end{thm}

In Theorem \ref{theoremdualunifieloc}, when $p\leq1$, we can take $0<\eta\leq1$ for $1<r\leq\infty$.

Also, for $0<q\leq1$, $q\leq p<\infty$, $\delta\geq\left\lfloor d\left(\frac{1}{q}-1\right)\right\rfloor$ and, $1\leq r<p'$ if $1<p$ or $1\leq r<\infty$ otherwise, where $\frac{1}{p}+\frac{1}{p'}=1$, with $0<\eta<q$ if $1<r$ or $0<\eta\leq1$ if $r=1$, we have 
\begin{eqnarray}
\left(\mathcal{H}_{\mathrm{loc}}^{(q,p)}\right)^{\ast}\cong\mathcal{L}_{r,\phi_1,\delta}^{(q,p,\eta)\mathrm{loc}}\hookrightarrow\mathcal{L}_{r,\phi_1,\delta}^{(q,p,\eta)}\cong\left(\mathcal{H}^{(q,p)}\right)^{\ast}, \label{inclusqp}
\end{eqnarray}
thanks to Theorems  \ref{theoremdualunifie} and \ref{theoremdualunifieloc}. Moreover, $\mathcal{L}_{r,\phi_1,\delta}^{(q,p,\eta)\mathrm{loc}}\subsetneq\mathcal{L}_{r,\phi_1,\delta}^{(q,p,\eta)}$, and hence 
\begin{align}
\left(\mathcal{H}_{\mathrm{loc}}^{(q,p)}\right)^{\ast}\subsetneq\left(\mathcal{H}^{(q,p)}\right)^{\ast}. \label{compardeshqp1}
\end{align}

To complete, we give some embedding relations relating to the spaces $\mathcal{L}_{r,\phi_1,\delta}^{(q,p,\eta)\mathrm{loc}}$ and similar to those obtained in the setting of the spaces $\mathcal{L}_{r,\phi_1,\delta}^{(q,p,\eta)}$ (see \cite[(3.17), (3.18), (3.19), Proposition 3.6 and (3.36)]{AbFt3}). Therefore, we leave the details of their proofs to the reader. 

For $0<q\leq1$, $0<p<\infty$, $1\leq r<\infty$ and $0<\eta_1\leq\eta_2<\infty$, we have  
\begin{align}
\left\|g\right\|_{\mathcal{L}_{r,\phi_{1},\delta}^{(q,p,\eta_1)\mathrm{loc}}}\leq\left\|g\right\|_{\mathcal{L}_{r,\phi_{1},\delta}^{(q,p,\eta_2)\mathrm{loc}}}, \label{revisiop07}
\end{align}
for all $g\in\mathcal{L}_{r,\phi_{1},\delta}^{(q,p,\eta_2)\mathrm{loc}}$, and hence $\mathcal{L}_{r,\phi_{1},\delta}^{(q,p,\eta_2)\mathrm{loc}}\hookrightarrow\mathcal{L}_{r,\phi_{1},\delta}^{(q,p,\eta_1)\mathrm{loc}}\hookrightarrow\mathcal{L}_{r,\phi_{1},\delta}^{\mathrm{loc}}$.

We have inequalities similar to (\ref{revisiop07}) with the exponents $q$ and $p$. More precisely, for  $0<q\leq q_1\leq1$, $0<p\leq p_1<\infty$, $1\leq r<\infty$ and $0<\eta<\infty$, with the functions $\phi_{1}(Q):=\frac{\left\|\chi_{_{Q}}\right\|_{q,p}}{|Q|}$, $\psi_{1}(Q):=\frac{\left\|\chi_{_{Q}}\right\|_{q_1,p}}{|Q|}$ and $\varphi_{1}(Q):=\frac{\left\|\chi_{_{Q}}\right\|_{q,p_1}}{|Q|}$ for all $Q\in\mathcal{Q}$, we have  
\begin{align}
\left\|g\right\|_{\mathcal{L}_{r,\psi_{1},\delta}^{(q_1,p,\eta)\mathrm{loc}}}\leq\left\|g\right\|_{\mathcal{L}_{r,\phi_{1},\delta}^{(q,p,\eta)\mathrm{loc}}}, \label{revisiop008}
\end{align}
for all $g\in\mathcal{L}_{r,\phi_{1},\delta}^{(q,p,\eta)\mathrm{loc}}$, and hence $\mathcal{L}_{r,\phi_{1},\delta}^{(q,p,\eta)\mathrm{loc}}\hookrightarrow\mathcal{L}_{r,\psi_{1},\delta}^{(q_1,p,\eta)\mathrm{loc}}$; and 
\begin{align}
\left\|g\right\|_{\mathcal{L}_{r,\phi_{1},\delta}^{(q,p,\eta)\mathrm{loc}}}\leq\left\|g\right\|_{\mathcal{L}_{r,\varphi_{1},\delta}^{(q,p_1,\eta)\mathrm{loc}}}, \label{revisiop009}
\end{align}
for all $g\in\mathcal{L}_{r,\varphi_{1},\delta}^{(q,p_1,\eta)\mathrm{loc}}$, and hence $\mathcal{L}_{r,\varphi_{1},\delta}^{(q,p_1,\eta)\mathrm{loc}}\hookrightarrow\mathcal{L}_{r,\phi_{1},\delta}^{(q,p,\eta)\mathrm{loc}}$. Notice that an inequality similar to (\ref{revisiop07}) holds also for the exponent $r$.\\

We can also define on $\mathcal{L}_{r,\phi_{1},\delta}^{(q,p,\eta)\mathrm{loc}}$ another norm $|||\cdot|||_{\mathcal{L}_{r,\phi_{1},\delta}^{(q,p,\eta)\mathrm{loc}}}$ equivalent to $\left\|\cdot\right\|_{\mathcal{L}_{r,\phi_{1},\delta}^{(q,p,\eta)\mathrm{loc}}}$, with $0<q\leq1$, $0<p<\infty$, $0<\eta<\infty$ and $1\leq r<\infty$. Indeed, given a function $g$ in $L_{\mathrm{loc}}^r$ and an open subset $\Omega$ such that $\Omega\neq\mathbb{R}^d$, we put 
\begin{align*}\widetilde{\textit{O}(g,\Omega,r,\delta)}^{\mathrm{loc}}&:=\sup\left[\sum_{|\widetilde{Q^n}|<1}\inf_{\mathfrak{p}\in\mathcal{P_{\delta}}}|\widetilde{Q^n}|^{\frac{1}{r'}}\left(\int_{\widetilde{Q^n}}\left|g(x)-\mathfrak{p}(x)\right|^{r}dx\right)^{\frac{1}{r}}\right. \\
 &+\left. \sum_{|\widetilde{Q^n}|\geq1}|\widetilde{Q^n}|^{\frac{1}{r'}}\left(\int_{\widetilde{Q^n}}|g(x)|^{r}dx\right)^{\frac{1}{r}}\right],
\end{align*} 
where $\frac{1}{r}+\frac{1}{r'}=1$ and the supremum is taken over all families of cubes $\left\{Q^n\right\}_{n\geq0}$ such that $Q^n\subset\Omega$, for all $n\geq0$ and $\sum_{n\geq0}\chi_{_{Q^n}}\leq K(d)$, with $\widetilde{Q^n}=C_0Q^n$, where $K(d)>1$ and $C_0>1$ are the same constants as in the definition of $\textit{O}(g,\Omega,r)$. We have the following proposition.  

\begin{prop}\label{dualqpequival0}
Suppose that $0<q\leq1$ and $0<p<\infty$. Let $0<\eta<\infty$ and $1\leq r<\infty$. Let $g$ be a function in $L_{\mathrm{loc}}^r$. Then $g\in\mathcal{L}_{r,\phi_{1},\delta}^{(q,p,\eta)\mathrm{loc}}$ if and only if there is a constant $C>0$ such that, for all families of open subsets $\left\{\Omega^j\right\}_{j\in\mathbb{Z}}$ with $\left\|\sum_{j\in\mathbb{Z}}2^{j\eta}\chi_{\Omega^j}\right\|_{\frac{q}{\eta},\frac{p}{\eta}}<\infty$, we have  
\begin{align}
\sum_{j\in\mathbb{Z}}2^j\widetilde{\textit{O}(g,\Omega^j,r,\delta)}^{\mathrm{loc}}\leq C\left\|\sum_{j\in\mathbb{Z}}2^{j\eta}\chi_{\Omega^j}\right\|_{\frac{q}{\eta},\frac{p}{\eta}}^{\frac{1}{\eta}}. \label{dualqpequival}
\end{align}  
Moreover, if we define $|||g|||_{\mathcal{L}_{r,\phi_{1},\delta}^{(q,p,\eta)\mathrm{loc}}}:=\inf\left\{C>0:\ C \text{ satisfies } (\ref{dualqpequival})\right\}$, then   
\begin{align}
|||g|||_{\mathcal{L}_{r,\phi_{1},\delta}^{(q,p,\eta)\mathrm{loc}}}\approx\left\|g\right\|_{\mathcal{L}_{r,\phi_{1},\delta}^{(q,p,\eta)\mathrm{loc}}} \label{dualqpequival1}
\end{align} 
and $|||\cdot|||_{\mathcal{L}_{r,\phi_{1},\delta}^{(q,p,\eta)\mathrm{loc}}}$ is a norm on $\mathcal{L}_{r,\phi_1,\delta}^{(q,p,\eta)\mathrm{loc}}$.
\end{prop}

As a consequence of Proposition \ref{dualqpequival0}, for $0\leq\delta_1\leq\delta_2$ being two integers, $0<q\leq1$, $0<p<\infty$, $0<\eta<\infty$ and $1\leq r<\infty$, we have    
\begin{align}
\left\|g\right\|_{\mathcal{L}_{r,\phi_{1},\delta_2}^{(q,p,\eta)\mathrm{loc}}}\lsim\left\|g\right\|_{\mathcal{L}_{r,\phi_{1},\delta_1}^{(q,p,\eta)\mathrm{loc}}}, \label{revisiop08}
\end{align}
for all $g\in\mathcal{L}_{r,\phi_{1},\delta_1}^{(q,p,\eta)\mathrm{loc}}$. Hence $\mathcal{L}_{r,\phi_{1},\delta_1}^{(q,p,\eta)\mathrm{loc}}\hookrightarrow\mathcal{L}_{r,\phi_{1},\delta_2}^{(q,p,\eta)\mathrm{loc}}\hookrightarrow\mathcal{L}_{r,\phi_{1},\delta_2}^{\mathrm{loc}}$.

\section{On the inclusion of $\mathcal{H}^{(1,p)}$ in $(L^1,\ell^p)$ and of $\mathcal{H}^{(q,p)}$ in $\mathcal{H}_{\mathrm{loc}}^{(q,p)}$}

\subsection{On the inclusion of $\mathcal{H}^{(1,p)}$ in $(L^1,\ell^p)$}

We prove that the inclusion of $\mathcal{H}^{(1,p)}$ in $(L^1,\ell^p)$, for $1\leq p<\infty$, is strict; generalizing the one of $\mathcal{H}^1$ in $L^1$. For $p=1$, we have $\mathcal{H}^{(1,1)}=\mathcal{H}^1$ and $(L^1,\ell^1)=L^1$, and it is well known that $\mathcal{H}^1\subsetneq L^1$. We so assume that $1<p<\infty$. Our argument is similar to the one of the proof of \cite[Theorem 3.6 (v), p. 26]{YZDYWY}.  

We know that $\mathcal{H}^{(1,p)}\subset(L^1,\ell^p)$ with
\begin{align}
\left\|f\right\|_{1,p}\leq\left\|f\right\|_{\mathcal{H}^{(1,p)}}, \label{Bairetheori2}
\end{align}
for all $f\in\mathcal{H}^{(1,p)}$, by \cite[Theorem 3.2, p. 1905]{AbFt}. Moreover, $(\mathcal{H}^{(1,p)},\left\|\cdot\right\|_{\mathcal{H}^{(1,p)}})$ and $((L^1,\ell^p),\left\|\cdot\right\|_{1,p})$ are Banach spaces. 

Suppose now that $\mathcal{H}^{(1,p)}=(L^1,\ell^p)$ as sets. Then $(\mathcal{H}^{(1,p)},\left\|\cdot\right\|_{1,p})$ is Banach space since $((L^1,\ell^p),\left\|\cdot\right\|_{1,p})$ is so. Thus, (\ref{Bairetheori2}) and the fact that $(\mathcal{H}^{(1,p)},\left\|\cdot\right\|_{\mathcal{H}^{(1,p)}})$ and $(\mathcal{H}^{(1,p)},\left\|\cdot\right\|_{1,p})$ are Banach spaces, imply that 
\begin{eqnarray}
\left\|\cdot\right\|_{\mathcal{H}^{(1,p)}}\approx\left\|\cdot\right\|_{1,p} \label{Bairetheori3}
\end{eqnarray}
on $\mathcal{H}^{(1,p)}$; in other words, there exists a constant $C>0$ such that  
\begin{eqnarray}
\left\|f\right\|_{1,p}\leq\left\|f\right\|_{\mathcal{H}^{(1,p)}}\leq C\left\|f\right\|_{1,p}, \label{Bairetheori4} 
\end{eqnarray}
for all $f\in\mathcal{H}^{(1,p)}$, by \cite[Corollary 2.12 (d), pp. 49-50]{WR1} or \cite[Corollary 2.8, p. 35]{HBZ1} or yet \cite[Remarque 5, p. 19]{HBZ}. From (\ref{Bairetheori3}), it comes that $(\mathcal{H}^{(1,p)})^{\ast}=(L^1,\ell^p)^{\ast}$; namely  $(L^{\infty},\ell^{p'})\cong\mathcal{L}_{1,\phi_1,\delta}^{(1,p,\eta_1)}\cong\mathcal{L}_{r,\phi_1,\delta}^{(1,p,\eta)}$ (see Remark \ref{remarqedualeqp1}). Indeed, for $T\in(\mathcal{H}^{(1,p)})^{\ast}$ and for all $f\in(L^1,\ell^p)$, we have 
\begin{eqnarray*}
\left|T(f)\right|\leq C\left\|f\right\|_{\mathcal{H}^{(1,p)}}\leq C\left\|f\right\|_{1,p},
\end{eqnarray*}
by (\ref{Bairetheori4}), since $(L^1,\ell^p)\subset\mathcal{H}^{(1,p)}$ by assumption; which implies that $T$ is a continuous linear functional on $(L^1,\ell^p)$; in other words $(\mathcal{H}^{(1,p)})^{\ast}\subset(L^1,\ell^p)^{\ast}$. But this is opposite to the fact that $(L^1,\ell^p)^{\ast}\subsetneq(\mathcal{H}^{(1,p)})^{\ast}$; namely $(L^{\infty},\ell^{p'})\subsetneq\mathcal{L}_{1,\phi_1,\delta}^{(1,p,\eta_1)}\cong\mathcal{L}_{r,\phi_1,\delta}^{(1,p,\eta)}$ (see Remark \ref{remarqedualeqp1}). Therefore, $\mathcal{H}^{(1,p)}\neq(L^1,\ell^p)$, and hence $\mathcal{H}^{(1,p)}\subsetneq(L^1,\ell^p)$. 

\begin{remark}
Notice that it is possible to deal with the cases $p=1$ and $1<p<\infty$ simultaneously by this method, since $(\mathcal{H}^{(1,p)})^{\ast}\cong\mathcal{L}_{1,\phi_1,\delta}^{(1,p,\eta_1)}\cong\mathcal{L}_{r,\phi_1,\delta}^{(1,p,\eta)}$ and $(L^{\infty},\ell^{p'})\subsetneq\mathcal{L}_{1,\phi_1,\delta}^{(1,p,\eta_1)}\cong\mathcal{L}_{r,\phi_1,\delta}^{(1,p,\eta)}$, for $1\leq p<\infty$, under the assumptions of Remark \ref{remarqedualeqp1}, by Theorem \ref{theoremdualunifie} (see \cite[Theorem 3.8]{AbFt3}).
\end{remark}

\subsection{On the inclusion of $\mathcal{H}^{(q,p)}$ in $\mathcal{H}_{\mathrm{loc}}^{(q,p)}$}

We assume that $0<q\leq1$ and $q\leq p<\infty$. The reasoning is similar to the one of the strict inclusion of $\mathcal{H}^{(1,p)}$ in $(L^1,\ell^p)$. We know that $\mathcal{H}^{(q,p)}\subset\mathcal{H}_{\mathrm{loc}}^{(q,p)}$ with
\begin{align}
\left\|f\right\|_{\mathcal{H}_{\mathrm{loc}}^{(q,p)}}\leq\left\|f\right\|_{\mathcal{H}^{(q,p)}}, \label{Bairetheor1} 
\end{align}
for all $f\in\mathcal{H}^{(q,p)}$, by definition. Also, $(\mathcal{H}^{(q,p)},\left\|\cdot\right\|_{\mathcal{H}^{(q,p)}})$ and $(\mathcal{H}_{\mathrm{loc}}^{(q,p)},\left\|\cdot\right\|_{\mathcal{H}_{\mathrm{loc}}^{(q,p)}})$ are quasi-Banach spaces, and hence Fr\'echet spaces (F-spaces) (for F-space, see \cite[Definition 1, p. 52]{KYos}). 

Suppose that $\mathcal{H}^{(q,p)}=\mathcal{H}_{\mathrm{loc}}^{(q,p)}$ as sets. Then $(\mathcal{H}^{(q,p)},\left\|\cdot\right\|_{\mathcal{H}_{\mathrm{loc}}^{(q,p)}})$ is a Fr\'echet space because $(\mathcal{H}_{\mathrm{loc}}^{(q,p)},\left\|\cdot\right\|_{\mathcal{H}_{\mathrm{loc}}^{(q,p)}})$ is so. Thus, (\ref{Bairetheor1}) and the fact that $(\mathcal{H}^{(q,p)},\left\|\cdot\right\|_{\mathcal{H}^{(q,p)}})$ and $(\mathcal{H}^{(q,p)},\left\|\cdot\right\|_{\mathcal{H}_{\mathrm{loc}}^{(q,p)}})$ are Fr\'echet spaces, imply that 
\begin{align}
\left\|\cdot\right\|_{\mathcal{H}_{\mathrm{loc}}^{(q,p)}}\approx\left\|\cdot\right\|_{\mathcal{H}^{(q,p)}} \label{Bairetheor2} 
\end{align}
on $\mathcal{H}^{(q,p)}$; in other words, there exists a constant $C>0$ sucht that 
\begin{align}
\left\|f\right\|_{\mathcal{H}_{\mathrm{loc}}^{(q,p)}}\leq\left\|f\right\|_{\mathcal{H}^{(q,p)}}\leq C\left\|f\right\|_{\mathcal{H}_{\mathrm{loc}}^{(q,p)}}, \label{Bairetheor3} 
\end{align}
for all $f\in\mathcal{H}^{(q,p)}$, by \cite[Corollary 2.12 (d), pp. 49-50]{WR1}. From (\ref{Bairetheor2}), it follows that $(\mathcal{H}^{(q,p)})^{\ast}=(\mathcal{H}_{\mathrm{loc}}^{(q,p)})^{\ast}$; namely $\mathcal{L}_{r,\phi_1,\delta}^{(q,p,\eta)}\cong\mathcal{L}_{r,\phi_1,\delta}^{(q,p,\eta)\mathrm{loc}}$ (see (\ref{inclusqp})). In fact, for $T\in(\mathcal{H}^{(q,p)})^{\ast}$ and for all $f\in\mathcal{H}_{\mathrm{loc}}^{(q,p)}$, we have 
\begin{eqnarray*}
\left|T(f)\right|\leq C\left\|f\right\|_{\mathcal{H}^{(q,p)}}\leq C\left\|f\right\|_{\mathcal{H}_{\mathrm{loc}}^{(q,p)}},
\end{eqnarray*}
by (\ref{Bairetheor3}), since $\mathcal{H}_{\mathrm{loc}}^{(q,p)}\subset\mathcal{H}^{(q,p)}$ by assumption; which implies that $T$ is a continuous linear functional on $\mathcal{H}_{\mathrm{loc}}^{(q,p)}$; in other words $(\mathcal{H}^{(q,p)})^{\ast}\subset(\mathcal{H}_{\mathrm{loc}}^{(q,p)})^{\ast}$. But this contradicts the fact that $(\mathcal{H}_{\mathrm{loc}}^{(q,p)})^{\ast}\subsetneq(\mathcal{H}^{(q,p)})^{\ast}$; namely $\mathcal{L}_{r,\phi_1,\delta}^{(q,p,\eta)\mathrm{loc}}\subsetneq\mathcal{L}_{r,\phi_1,\delta}^{(q,p,\eta)}$ (see (\ref{compardeshqp1})). Consequently, $\mathcal{H}^{(q,p)}\neq\mathcal{H}_{\mathrm{loc}}^{(q,p)}$, and hence $\mathcal{H}^{(q,p)}\subsetneq\mathcal{H}_{\mathrm{loc}}^{(q,p)}$.

\begin{remark}
We point out that $\mathcal{H}^{(1,p)}\subsetneq(L^1,\ell^p)$, for $1\leq p<\infty$, can be deduced from $\mathcal{H}^{(q,p)}\subsetneq\mathcal{H}_{\mathrm{loc}}^{(q,p)}$. Indeed, we have $\mathcal{H}^{(1,p)}\subsetneq\mathcal{H}_{\mathrm{loc}}^{(1,p)}\subset(L^1,\ell^p)$, by \cite[Theorem 3.2, p. 1905]{AbFt}, and hence $\mathcal{H}^{(1,p)}\subsetneq(L^1,\ell^p)$.

Also, although $\mathcal{H}^{(q,p)}\subset\mathcal{H}_{\mathrm{loc}}^{(q,p)}$ with $\left\|f\right\|_{\mathcal{H}_{\mathrm{loc}}^{(q,p)}}\leq\left\|f\right\|_{\mathcal{H}^{(q,p)}}$, for all $f\in\mathcal{H}^{(q,p)}$, and $(\mathcal{H}^{(q,p)},\left\|\cdot\right\|_{\mathcal{H}^{(q,p)}})$ and $(\mathcal{H}_{\mathrm{loc}}^{(q,p)},\left\|\cdot\right\|_{\mathcal{H}_{\mathrm{loc}}^{(q,p)}})$ are Fr\'echet spaces for $0<q\leq1$ and $0<p<q\leq1$, it is not yet clear that $\mathcal{H}^{(q,p)}\subsetneq\mathcal{H}_{\mathrm{loc}}^{(q,p)}$ for this case.

On the other hand, contrary to $\mathcal{H}^{(1,p)}\subsetneq(L^1,\ell^p)$, it is not clear that $\mathcal{H}_{\mathrm{loc}}^{(1,p)}\subsetneq(L^1,\ell^p)$. In fact, although $(L^{\infty},\ell^{p'})\subset\mathcal{L}_{1,\phi_1,\delta}^{(1,p,\eta_1)\mathrm{loc}}\cong\mathcal{L}_{r,\phi_1,\delta}^{(1,p,\eta)\mathrm{loc}}$, by Remark \ref{remarqedualeqploc1}, it is not clear that this inclusion is strict. The difficulty is that, contrary to the proof of $(L^{\infty},\ell^{p'})\subsetneq\mathcal{L}_{1,\phi_1,\delta}^{(1,p,\eta_1)}\cong\mathcal{L}_{r,\phi_1,\delta}^{(1,p,\eta)}$, we have $(L^{\infty},\ell^{p'})\subset L^{\infty}$, $\mathcal{P}_{\delta}\cap L^{\infty}=\mathbb{C}=\mathcal{P}_{\delta}\cap\mathcal{L}_{1,\phi_1,\delta}^{(1,1,\eta_1)\mathrm{loc}}=\mathcal{P}_{\delta}\cap\mathcal{L}_{r,\phi_1,\delta}^{(1,1,\eta)\mathrm{loc}}$ and $\mathcal{P}_{\delta}\cap\mathcal{L}_{1,\phi_1,\delta}^{(1,p,\eta_1)\mathrm{loc}}=\mathcal{P}_{\delta}\cap\mathcal{L}_{r,\phi_1,\delta}^{(1,p,\eta)\mathrm{loc}}=\left\{0\right\}$ for $1<p<\infty$. Hence we can not conclude that $(L^{\infty},\ell^{p'})\subsetneq\mathcal{L}_{1,\phi_1,\delta}^{(1,p,\eta_1)\mathrm{loc}}\cong\mathcal{L}_{r,\phi_1,\delta}^{(1,p,\eta)\mathrm{loc}}$ for $1\leq p<\infty$. 

More generally, contrary to $\mathcal{P}_{\delta}\subset\mathcal{L}_{r,\phi_1,\delta}^{(q,p,\eta)}$ for $0<q\leq 1$, $0<p<\infty$, $1\leq r<\infty$, $0<\eta<\infty$ and $\delta\geq0$, we have for $\delta\geq\max\left\{\left\lfloor d\left(\frac{1}{q}-1\right)\right\rfloor,\left\lfloor d\left(\frac{1}{p}-1\right)\right\rfloor\right\}$,
\begin{align}
\mathcal{L}_{r,\phi_1,\delta}^{(q,p,\eta)\mathrm{loc}}\cap\mathcal{P}_{\delta}=\mathbb{C}, \label{Bairetheor5}
\end{align}
for $0<q,p\leq 1$ with $q\neq1$ or $p\neq1$, and
\begin{align}
\mathcal{L}_{r,\phi_1,\delta}^{(q,p,\eta)\mathrm{loc}}\cap\mathcal{P}_{\delta}=\left\{0\right\}, \label{Bairetheor6}
\end{align}
for $0<q\leq 1$ and $1<p<\infty$, with $1\leq r<p'$ if $1<p$ or $1\leq r<\infty$ if $p\leq1$, and $0<\eta<q$ if $1<r$ or $0<\eta\leq1$ if $r=1$. When $q=p=1$, we have for $\delta\geq0$,   
\begin{align}
\mathcal{L}_{r,\phi_1,\delta}^{(1,1,\eta)\mathrm{loc}}\cap\mathcal{P}_{\delta}=\mathbb{C}, \label{Bairetheor7}
\end{align} 
and hence $bmo(\mathbb{R}^d)\cap\mathcal{P}_{\delta}=\mathbb{C}$. 
\end{remark}

Let us give the proofs of (\ref{Bairetheor5}), (\ref{Bairetheor6}) and (\ref{Bairetheor7}). We shall need the following well-known estimates whose we shall give a proof for the reader's convenience. For $0<q,p<\infty$, we have 
\begin{align}
\left\|\chi_{Q}\right\|_{q,p}\approx\left\|\chi_{Q}\right\|_{p}, \label{Bairetheor8}
\end{align}
for all cubes $Q$ such that $|Q|\geq1$, and 
\begin{align}
\left\|\chi_{Q}\right\|_{q,p}\approx\left\|\chi_{Q}\right\|_{q}, \label{Bairetheor9}
\end{align}
for all cubes $Q$ such that $|Q|\leq1$. The proofs of (\ref{Bairetheor8}) and (\ref{Bairetheor9}) will be given after. 

Proof of (\ref{Bairetheor5}). Let $c\in\mathbb{C}$. Consider the constant function $g=c$. We have $g\in\mathcal{P}_{\delta}$ and, according to (\ref{Bairetheor8}),    
\begin{align*}
\left\|g\right\|_{\mathcal{L}_{r,\phi_1,\delta}^{\mathrm{loc}}}&=\sup_{\underset{|Q|\geq1}{Q\in\mathcal{Q}}}\frac{|Q|}{\left\|\chi_{Q}\right\|_{q,p}}\left(\frac{1}{|Q|}\int_{Q}|g(x)|^rdx\right)^{\frac{1}{r}}\\
&\approx\sup_{\underset{|Q|\geq1}{Q\in\mathcal{Q}}}\frac{|Q|}{\left\|\chi_{Q}\right\|_{p}}\left(\frac{1}{|Q|}\int_{Q}|g(x)|^rdx\right)^{\frac{1}{r}}\\
&=\sup_{\underset{|Q|\geq1}{Q\in\mathcal{Q}}}|Q|^{1-\frac{1}{p}}\left(\frac{1}{|Q|}\int_{Q}|g(x)|^rdx\right)^{\frac{1}{r}}=|c|\sup_{\underset{|Q|\geq1}{Q\in\mathcal{Q}}}|Q|^{1-\frac{1}{p}}=|c|<\infty,
\end{align*}
because $1-\frac{1}{p}\leq0$ implies that $0<|Q|^{1-\frac{1}{p}}\leq1$ for all cubes $Q$ such that $|Q|\geq1$. Hence $g=c\in\mathcal{L}_{r,\phi_1,\delta}^{\mathrm{loc}}\cong\mathcal{L}_{r,\phi_1,\delta}^{(q,p,\eta)\mathrm{loc}}$, by (\ref{dualqploc3bisbis}). Therefore, $\mathbb{C}\subset\mathcal{L}_{r,\phi_1,\delta}^{(q,p,\eta)\mathrm{loc}}\cap\mathcal{P}_{\delta}$.   

For the converse inclusion; namely $\mathcal{L}_{r,\phi_1,\delta}^{(q,p,\eta)\mathrm{loc}}\cap\mathcal{P}_{\delta}\subset\mathbb{C}$, 
we have $$\mathcal{L}_{r,\phi_1,\delta}^{(q,p,\eta)\mathrm{loc}}\cap\mathcal{P}_{\delta}\subset\mathcal{L}_{r,\phi_2,\delta}^{\mathrm{loc}}\cap\mathcal{P}_{\delta}=\Lambda_{d\left(\frac{1}{q}-1\right)}\cap\mathcal{P}_{\delta}\subset L^{\infty}\cap\mathcal{P}_{\delta}=\mathbb{C},$$ when $0<q\leq p\leq 1$ and $q\neq1$, because $\mathcal{L}_{r,\phi_1,\delta}^{(q,p,\eta)\mathrm{loc}}\subset\mathcal{L}_{r,\phi_2,\delta}^{\mathrm{loc}}=\Lambda_{d\left(\frac{1}{q}-1\right)}\subset L^{\infty}$, by (\ref{dualqploc3bbi}), where $\Lambda_{d\left(\frac{1}{q}-1\right)}$ is the dual space of $\mathcal{H}_{\mathrm{loc}}^q$ defined by D. Goldberg \cite{DGG}. On the other hand, we have $$\mathcal{L}_{r,\phi_1,\delta}^{(q,p,\eta)\mathrm{loc}}\cap\mathcal{P}_{\delta}=\mathcal{L}_{r,\phi_3,\delta}^{\mathrm{loc}}\cap\mathcal{P}_{\delta}=\Lambda_{d\left(\frac{1}{p}-1\right)}\cap\mathcal{P}_{\delta}\subset L^{\infty}\cap\mathcal{P}_{\delta}=\mathbb{C},$$ when $0<p,q\leq1$ and $p\neq1$, because for all $g\in\mathcal{L}_{r,\phi_1,\delta}^{(q,p,\eta)\mathrm{loc}}\cap\mathcal{P}_{\delta}$ (or $g\in\mathcal{L}_{r,\phi_3,\delta}^{\mathrm{loc}}\cap\mathcal{P}_{\delta}$),  
\begin{align*}
\left\|g\right\|_{\mathcal{L}_{r,\phi_1,\delta}^{(q,p,\eta)\mathrm{loc}}}\approx\left\|g\right\|_{\mathcal{L}_{r,\phi_1,\delta}^{\mathrm{loc}}}&=\sup_{\underset{|Q|\geq1}{Q\in\mathcal{Q}}}\frac{|Q|}{\left\|\chi_{Q}\right\|_{q,p}}\left(\frac{1}{|Q|}\int_{Q}|g(x)|^rdx\right)^{\frac{1}{r}}\\
&\approx\sup_{\underset{|Q|\geq1}{Q\in\mathcal{Q}}}\frac{|Q|}{\left\|\chi_{Q}\right\|_{p}}\left(\frac{1}{|Q|}\int_{Q}|g(x)|^rdx\right)^{\frac{1}{r}}=\left\|g\right\|_{\mathcal{L}_{r,\phi_3,\delta}^{\mathrm{loc}}},
\end{align*}
and $\mathcal{L}_{r,\phi_3,\delta}^{\mathrm{loc}}=\Lambda_{d\left(\frac{1}{p}-1\right)}\subset L^{\infty}$ (since $\delta\geq\left\lfloor d\left(\frac{1}{p}-1\right)\right\rfloor$). This establishes (\ref{Bairetheor5}).   

For (\ref{Bairetheor6}), it is clear that $\left\{0\right\}\subset\mathcal{L}_{r,\phi_1,\delta}^{(q,p,\eta)\mathrm{loc}}\cap\mathcal{P}_{\delta}$. Conversely, let $g\in\mathcal{L}_{r,\phi_1,\delta}^{(q,p,\eta)\mathrm{loc}}\cap\mathcal{P}_{\delta}$. We have $g\in L^{p'}\cap\mathcal{P}_{\delta}$, by Remark \ref{remarqedualeqploc1}. Hence necessarily $g$ is a null polynomial, because all non-null polynomials do not belong to $L^{p'}$, $1<p'<\infty$. Thus, $\mathcal{L}_{r,\phi_1,\delta}^{(q,p,\eta)\mathrm{loc}}\cap\mathcal{P}_{\delta}\subset\left\{0\right\}$, which proves (\ref{Bairetheor6}). 

For the proof of (\ref{Bairetheor7}), let $0<p<1$. If $\delta\geq\left\lfloor d\left(\frac{1}{p}-1\right)\right\rfloor$, then we have $\mathcal{L}_{r,\phi_1,\delta}^{(1,1,\eta)\mathrm{loc}}\cap\mathcal{P}_{\delta}\subset\mathcal{L}_{r,\varphi_1,\delta}^{(1,p,\eta)\mathrm{loc}}\cap\mathcal{P}_{\delta}\subset\mathbb{C}$, by (\ref{revisiop009}) and (\ref{Bairetheor5}). If $\delta\leq\left\lfloor d\left(\frac{1}{p}-1\right)\right\rfloor$, then we have $\mathcal{L}_{r,\phi_1,\delta}^{(1,1,\eta)\mathrm{loc}}\cap\mathcal{P}_{\delta}\subset\mathcal{L}_{r,\varphi_1,\delta}^{(1,p,\eta)\mathrm{loc}}\cap\mathcal{P}_{\delta}\subset\mathcal{L}_{r,\varphi_1,\left\lfloor d\left(\frac{1}{p}-1\right)\right\rfloor}^{(1,p,\eta)\mathrm{loc}}\cap\mathcal{P}_{\left\lfloor d\left(\frac{1}{p}-1\right)\right\rfloor}\subset\mathbb{C}$, by (\ref{revisiop009}), (\ref{revisiop08}) and (\ref{Bairetheor5}). For the converse inclusion; namely $\mathbb{C}\subset\mathcal{L}_{r,\phi_1,\delta}^{(1,1,\eta)\mathrm{loc}}\cap\mathcal{P}_{\delta}$, see the first part of the proof of (\ref{Bairetheor5}), and hence (\ref{Bairetheor7}) is proved.\\

Now we give the proofs of Estimates (\ref{Bairetheor8}) and (\ref{Bairetheor9}). Let $Q$ be a cube and $\ell_Q$ be its side-length. Without loss generality, we can assume that $Q$ is closed. We have  
\begin{equation}
1\leq M_Q:=\sharp{\left\{k\in\mathbb{Z}^d:\ Q\cap Q_k\neq\emptyset\right\}}<(\ell_Q+2)^d\leq\left\{\begin{array}{lll}3^d&\text{ if }&|Q|\leq 1\\ \\
3^d\ell_Q^d&\text{ if }&|Q|\geq1,\end{array}\right. \label{Bairetheor10}
\end{equation}
where $\sharp$ denotes the cardinal. To see (\ref{Bairetheor10}), denote by $x^Q=(x_1^Q,x_2^Q,\ldots,x_d^Q)$ the center of $Q$. Let $k\in\mathbb{Z}^d$ such that $Q\cap Q_k\neq\emptyset$ (a such $k$ exists because $\left\{Q_k\right\}_{k\in\mathbb{Z}^d}$ is a partition of $\mathbb{R}^d$, and for recall $Q_k=k+[0,1)^{d}$). Then there exists $x=(x_1,x_2,\ldots,x_d)\in\mathbb{R}^d$ such that $x\in Q\cap Q_k$, and hence satisfies  
\begin{align}
x_i^Q-\frac{\ell_Q}{2}\leq x_i\leq\frac{\ell_Q}{2}+x_i^Q\ \text{ and }\ k_i\leq x_i<k_i+1, \label{Afterwa}
\end{align}
for every $i=1,2,\ldots,d$, where $k_i$'s are the coordinates of $k$. But (\ref{Afterwa}) implies that  
\begin{align*} 
x_i^Q-\frac{\ell_Q}{2}-1<k_i\leq x_i^Q+\frac{\ell_Q}{2}\ ,
\end{align*}
for every $i=1,2,\ldots,d$. Since $k_i\in\mathbb{Z}$, it follows that 
\begin{align*} 
\left\lfloor x_i^Q-\frac{\ell_Q}{2}-1\right\rfloor+1\leq k_i\leq \left\lfloor x_i^Q+\frac{\ell_Q}{2}\right\rfloor\ ,
\end{align*}
for every $i=1,2,\ldots,d$. Thus, for fixed $i$, denoting by $n(k_i)$ the number of possible values that $k_i$ can take, we have 
\begin{align*} 
n(k_i)&=\left\lfloor x_i^Q+\frac{\ell_Q}{2}\right\rfloor-\left(\left\lfloor x_i^Q-\frac{\ell_Q}{2}-1\right\rfloor+1\right)+1\\
&<x_i^Q+\frac{\ell_Q}{2}-\left(x_i^Q-\frac{\ell_Q}{2}-1\right)+1\\
&=\ell_Q+2. 
\end{align*}
Hence $$1\leq\sharp{\left\{k\in\mathbb{Z}^d:\ Q\cap Q_k\neq\emptyset\right\}}<(\ell_Q+2)^d.$$ Moreover, if $|Q|\leq 1$, then $\ell_Q\leq1$, and hence $(\ell_Q+2)^d\leq 3^d$. If $|Q|\geq1$, then $\ell_Q\geq1$, and hence $(\ell_Q+2)^d\leq(\ell_Q+2\ell_Q)^d=3^d\ell_Q^d$. This establishes (\ref{Bairetheor10}).\\ 
Also, we have 
\begin{align}
\left\|\chi_{Q}\right\|_{q,p}^p=\sum_{k\in\mathbb{Z}^d}|Q\cap Q_k|^{\frac{p}{q}}=\sum_{\underset{Q\cap Q_k\neq\emptyset}{k\in\mathbb{Z}^d}}|Q\cap Q_k|^{\frac{p}{q}}. \label{Afterwa1}
\end{align}

Assume that $\frac{p}{q}\leq1$ (ie $q\geq p$). Then we have 
\begin{align*}
|Q|^{\frac{p}{q}}=\left(\sum_{\underset{Q\cap Q_k\neq\emptyset}{k\in\mathbb{Z}^d}}|Q\cap Q_k|\right)^{\frac{p}{q}}\leq\sum_{\underset{Q\cap Q_k\neq\emptyset}{k\in\mathbb{Z}^d}}|Q\cap Q_k|^{\frac{p}{q}}&\leq M_Q^{-\frac{\frac{p}{q}}{\left(\frac{p}{q}\right)'}}\left(\sum_{\underset{Q\cap Q_k\neq\emptyset}{k\in\mathbb{Z}^d}}|Q\cap Q_k|\right)^{\frac{p}{q}}\\
&=M_Q^{-\frac{\frac{p}{q}}{\left(\frac{p}{q}\right)'}}|Q|^{\frac{p}{q}},
\end{align*}
by \cite[Proposition 2.1, p. 311]{RBHS}, and hence 
\begin{align*}
\left\|\chi_{Q}\right\|_q^p=|Q|^{\frac{p}{q}}\leq\left\|\chi_{Q}\right\|_{q,p}^p\leq M_Q^{-\frac{\frac{p}{q}}{\left(\frac{p}{q}\right)'}}|Q|^{\frac{p}{q}}&=M_Q^{-\frac{\frac{p}{q}}{\left(\frac{p}{q}\right)'}}\left\|\chi_{Q}\right\|_q^p\\
&=M_Q^{1-\frac{p}{q}}\left\|\chi_{Q}\right\|_q^p\leq(\ell_Q+2)^{d\left(1-\frac{p}{q}\right)}\left\|\chi_{Q}\right\|_q^p,
\end{align*}
according to (\ref{Afterwa1}) and (\ref{Bairetheor10}). Therefore, 
\begin{align*}
\left\|\chi_{Q}\right\|_q\leq\left\|\chi_{Q}\right\|_{q,p}\leq(\ell_Q+2)^{d\left(\frac{1}{p}-\frac{1}{q}\right)}\left\|\chi_{Q}\right\|_q.
\end{align*}
Thus: 

$\bullet$ If $|Q|\leq1$, then
\begin{equation}
\left\|\chi_{Q}\right\|_q\leq\left\|\chi_{Q}\right\|_{q,p}\leq(\ell_Q+2)^{d\left(\frac{1}{p}-\frac{1}{q}\right)}\left\|\chi_{Q}\right\|_q\leq3^{d\left(\frac{1}{p}-\frac{1}{q}\right)}\left\|\chi_{Q}\right\|_q, \label{Bairetheor11}
\end{equation}
by (\ref{Bairetheor10}). This states (\ref{Bairetheor9}) when $q\geq p$.

$\bullet$ If $|Q|\geq1$, then  
\begin{align*}
\left\|\chi_{Q}\right\|_q\leq\left\|\chi_{Q}\right\|_{q,p}\leq(\ell_Q+2)^{d\left(\frac{1}{p}-\frac{1}{q}\right)}\left\|\chi_{Q}\right\|_q&\leq3^{d\left(\frac{1}{p}-\frac{1}{q}\right)}\ell_Q^{d\left(\frac{1}{p}-\frac{1}{q}\right)}\left\|\chi_{Q}\right\|_q\\
&=3^{d\left(\frac{1}{p}-\frac{1}{q}\right)}\ell_Q^{\frac{d}{p}}=3^{d\left(\frac{1}{p}-\frac{1}{q}\right)}\left\|\chi_{Q}\right\|_p,
\end{align*}
by (\ref{Bairetheor10}). Furthermore, $\left\|\chi_{Q}\right\|_p\leq\left\|\chi_{Q}\right\|_{q,p}$, when $q\geq p$. Hence 
\begin{equation}
\left\|\chi_{Q}\right\|_p\leq\left\|\chi_{Q}\right\|_{q,p}\leq3^{d\left(\frac{1}{p}-\frac{1}{q}\right)}\left\|\chi_{Q}\right\|_p. \label{Bairetheor12}
\end{equation}
This states (\ref{Bairetheor8}) when $q\geq p$.

Assume that $\frac{p}{q}\geq1$ (ie $q\leq p$). Then we have  
\begin{align*}
M_Q^{-\frac{\frac{p}{q}}{\left(\frac{p}{q}\right)'}}|Q|^{\frac{p}{q}}=M_Q^{-\frac{\frac{p}{q}}{\left(\frac{p}{q}\right)'}}\left(\sum_{\underset{Q\cap Q_k\neq\emptyset}{k\in\mathbb{Z}^d}}|Q\cap Q_k|\right)^{\frac{p}{q}}&\leq\sum_{\underset{Q\cap Q_k\neq\emptyset}{k\in\mathbb{Z}^d}}|Q\cap Q_k|^{\frac{p}{q}}\\
&\leq\left(\sum_{\underset{Q\cap Q_k\neq\emptyset}{k\in\mathbb{Z}^d}}|Q\cap Q_k|\right)^{\frac{p}{q}}=|Q|^{\frac{p}{q}},
\end{align*}
by \cite[Proposition 2.1, p. 311]{RBHS}, and hence 
\begin{align*}
(\ell_Q+2)^{d\left(1-\frac{p}{q}\right)}\left\|\chi_{Q}\right\|_q^p\leq M_Q^{1-\frac{p}{q}}\left\|\chi_{Q}\right\|_q^p&=M_Q^{-\frac{\frac{p}{q}}{\left(\frac{p}{q}\right)'}}\left\|\chi_{Q}\right\|_q^p\\
&=M_Q^{-\frac{\frac{p}{q}}{\left(\frac{p}{q}\right)'}}|Q|^{\frac{p}{q}}\leq\left\|\chi_{Q}\right\|_{q,p}^p\leq|Q|^{\frac{p}{q}}=\left\|\chi_{Q}\right\|_q^p,
\end{align*}
according to (\ref{Bairetheor10}) and (\ref{Afterwa1}). Consequently, 
\begin{eqnarray*}
(\ell_Q+2)^{d\left(\frac{1}{p}-\frac{1}{q}\right)}\left\|\chi_{Q}\right\|_q\leq\left\|\chi_{Q}\right\|_{q,p}\leq\left\|\chi_{Q}\right\|_q.
\end{eqnarray*}
Thus:

$\bullet$ If $|Q|\leq1$, then
\begin{equation}
3^{d\left(\frac{1}{p}-\frac{1}{q}\right)}\left\|\chi_{Q}\right\|_q\leq(\ell_Q+2)^{d\left(\frac{1}{p}-\frac{1}{q}\right)}\left\|\chi_{Q}\right\|_q\leq\left\|\chi_{Q}\right\|_{q,p}\leq\left\|\chi_{Q}\right\|_q, \label{Bairetheor13}
\end{equation}
by (\ref{Bairetheor10}). This establishes (\ref{Bairetheor9}) when $q\leq p$.

$\bullet$ If $|Q|\geq1$, then 
\begin{align*}
3^{d\left(\frac{1}{p}-\frac{1}{q}\right)}\left\|\chi_{Q}\right\|_p=3^{d\left(\frac{1}{p}-\frac{1}{q}\right)}\ell_Q^{\frac{d}{p}}&=3^{d\left(\frac{1}{p}-\frac{1}{q}\right)}\ell_Q^{d\left(\frac{1}{p}-\frac{1}{q}\right)}\left\|\chi_{Q}\right\|_q\\
&\leq(\ell_Q+2)^{d\left(\frac{1}{p}-\frac{1}{q}\right)}\left\|\chi_{Q}\right\|_q\leq\left\|\chi_{Q}\right\|_{q,p}\leq\left\|\chi_{Q}\right\|_q,
\end{align*}
by (\ref{Bairetheor10}). Moreover, $\left\|\chi_{Q}\right\|_{q,p}\leq\left\|\chi_{Q}\right\|_p$, when $q\leq p$. Hence
\begin{equation}
3^{d\left(\frac{1}{p}-\frac{1}{q}\right)}\left\|\chi_{Q}\right\|_p\leq\left\|\chi_{Q}\right\|_{q,p}\leq\left\|\chi_{Q}\right\|_p. \label{Bairetheor14}
\end{equation}
This establishes (\ref{Bairetheor8}) when $q\leq p$. 

To sum up, for $0<q,p<\infty$, we have\\
$\bullet$ when $|Q|\leq1$, 
\begin{equation*}
\min\left\{1,3^{d\left(\frac{1}{p}-\frac{1}{q}\right)}\right\}\left\|\chi_{Q}\right\|_q\leq\left\|\chi_{Q}\right\|_{q,p}\leq\max\left\{1,3^{d\left(\frac{1}{p}-\frac{1}{q}\right)}\right\}\left\|\chi_{Q}\right\|_q,
\end{equation*}
by (\ref{Bairetheor11}) and (\ref{Bairetheor13}), which gives (\ref{Bairetheor9}).\\
$\bullet$ when $|Q|\geq1$, 
\begin{equation*}
\min\left\{1,3^{d\left(\frac{1}{p}-\frac{1}{q}\right)}\right\}\left\|\chi_{Q}\right\|_p\leq\left\|\chi_{Q}\right\|_{q,p}\leq\max\left\{1,3^{d\left(\frac{1}{p}-\frac{1}{q}\right)}\right\}\left\|\chi_{Q}\right\|_p,
\end{equation*}
by (\ref{Bairetheor12}) and (\ref{Bairetheor14}), which gives (\ref{Bairetheor8}).

\begin{remark}
We point out that (\ref{Bairetheor5}) is valid for $0<p,q\leq1$, $1\leq r<\infty$, $0<\eta\leq1$ and $\delta\geq0$, by (\ref{revisiop08}) and (\ref{Bairetheor7}), meanwhile (\ref{Bairetheor6}) is valid for $1\leq r<\infty$, $0<\eta<\infty$ and $\delta\geq0$, according to (\ref{revisiop07}) and (\ref{revisiop08}).
\end{remark}

\section{Boundedness of some classical operators on the dual spaces of Hardy-amalgam spaces}

\subsection{Convolution Operator}
Given a function $k$ defined and locally integrable on $\mathbb{R}^d\backslash\left\{0\right\}$, we say that a tempered distribution $K$ in $\mathbb{R}^d$ ($K\in\mathcal{S'}:=\mathcal{S'}(\mathbb{R}^d)$) coincides with the function $k$ on $\mathbb{R}^d\backslash\left\{0\right\}$, if 
\begin{eqnarray}
\left\langle K, \psi\right\rangle=\int_{\mathbb{R}^d}k(x)\psi(x)dx,\label{applicattheo090}
\end{eqnarray}
for all $\psi\in\mathcal{S}$, with $\text{supp}(\psi)\subset\mathbb{R}^d\backslash\left\{0\right\}$. Here, we are interested in tempered distributions $K$ in $\mathbb{R}^d$ that coincide with a function $k$ on $\mathbb{R}^d\backslash\left\{0\right\}$ and that have the form 
\begin{eqnarray}
\left\langle K, \psi\right\rangle=\lim_{j\rightarrow\infty}\int_{|x|\geq\sigma_j}k(x)\psi(x)dx,\ \ \ \psi\in\mathcal{S},  \label{applicattheo091}
\end{eqnarray}
for some sequence $\sigma_j\downarrow 0$ as $j\rightarrow\infty$ and independent of $\psi$. Also, we consider the convolution operators $T$: $T(f)=K\ast f$, $f\in\mathcal{S}$. Thus, when $\widehat{K}\in L^{\infty}$, we have 
\begin{eqnarray}
T(f)(x)=\int_{\mathbb{R}^d}k(x-y)f(y)dy, \label{applicattheo093}
\end{eqnarray}
for all $f\in L^2$ with compact support and all $x\notin\text{supp}(f)$. For (\ref{applicattheo093}), see \cite[Chap. 3, p. 113]{MA}. From now on, the letter $K$ stands for both the tempered distribution $K$ and the associated function $k$. We recall the following theorem.

\begin{thm}\cite[Theorem 5.1]{JD} \label{theoremsingaj}  Let $K$ be a tempered distribution in $\mathbb{R}^d$ which coincides with a locally integrable function on $\mathbb{R}^d\backslash\left\{0\right\}$ and is such that 
\begin{eqnarray*}
|\widehat{K}(\xi)|\leq A,
\end{eqnarray*} 
\begin{eqnarray*}
\int_{|x|>2|y|}|K(x-y)-K(x)|dx\leq B,\ y\in\mathbb{R}^d. 
\end{eqnarray*}
then, for $1<r<\infty$,
\begin{eqnarray*}
\left\|K\ast f\right\|_r\leq C_r\left\|f\right\|_r
\end{eqnarray*} 
and
\begin{eqnarray*}
\left|\left\{x\in\mathbb{R}^d:\ |K\ast f(x)|>\lambda\right\}\right|\leq\frac{C}{\lambda}\left\|f\right\|_1.
\end{eqnarray*} 
\end{thm}

Notice that $C_r:=C(d,r,A,B)$ (see \cite{JD}, p. 110). We proved the following results in \cite{AbFt1}. 

\begin{thm}\cite[Theorem 4.13]{AbFt1} \label{theoremsing4}
Let $K$ be a tempered distribution in $\mathbb{R}^d$ that coincides with a locally integrable function on $\mathbb{R}^d\backslash\left\{0\right\}$ and is such that 
\begin{eqnarray}
|\widehat{K}(\xi)|\leq A, \label{applicattheo12}
\end{eqnarray} 
and, there exist an integer $\delta>0$ and a constant $B>0$ such that
\begin{eqnarray}
|\partial^{\beta}K(x)|\leq\frac{B}{|x|^{d+|\beta|}}\ , \label{applicattheo15}
\end{eqnarray}
for all $x\neq0$ and all multi-indexes $\beta$ with $|\beta|\leq\delta$. If $\frac{d}{d+\delta}<q\leq 1$ and $q\leq p<\infty$, then the operator $T(f)=K\ast f$, for all $f\in\mathcal{S}$, extends to a bounded operator from $\mathcal{H}^{(q,p)}$ to $(L^q,\ell^p)$.
\end{thm}

\begin{thm}\cite[Theorem 4.17]{AbFt1} \label{theoremsing5}
Let $K$ be a tempered distribution in $\mathbb{R}^d$ that coincides with a locally integrable function on $\mathbb{R}^d\backslash\left\{0\right\}$ and satisfies assumptions (\ref{applicattheo12}) and (\ref{applicattheo15}). If $\frac{d}{d+\delta}<q\leq 1$ and $q\leq p<\infty$, then the operator $T(f)=K\ast f$, for all $f\in\mathcal{S}$, extends to a bounded operator from $\mathcal{H}^{(q,p)}$ to $\mathcal{H}^{(q,p)}$.  
\end{thm}

Our results are the following.

\begin{thm} \label{applicadualqpconv2} 
Suppose that $1\leq p<\infty$. Let $1\leq r<p'$ and, $0<\eta<1$ if $1<r<p'$ or $0<\eta\leq1$ if $r=1$. Let $K$ be a tempered distribution in $\mathbb{R}^d$ that coincides with a locally integrable function on $\mathbb{R}^d\backslash\left\{0\right\}$ and satisfies assumptions (\ref{applicattheo12}) and (\ref{applicattheo15}). 

Then the operator $T(f)=K\ast f$, for all $f\in\mathcal{S}$, is extendable on $(L^{\infty},\ell^{p'})$ and there exists a constant $C>0$ such that
\begin{eqnarray}
\left\|T(f)\right\|_{\mathcal{L}_{r,\phi_{1},\delta}^{(1,p,\eta)}}\leq C\left\|f\right\|_{\infty,p'}, \label{bisapplicattheo15}
\end{eqnarray}
for all $f\in(L^{\infty},\ell^{p'})$.
\end{thm}
\begin{proof}
If $p=1$, then $\mathcal{L}_{r,\phi_{1},\delta}^{(1,1,\eta)}\cong\mathcal{L}_{r,\phi_{1},\delta}\cong\mathcal{L}_{1,\phi_{1},0}=\mathcal{L}_{1,\phi_{2},0}=\mathrm{BMO}(\mathbb{R}^d)$ et $(L^{\infty},\ell^{1'})=(L^{\infty},\ell^{\infty})=L^{\infty}$, and hence (\ref{bisapplicattheo15}) holds by \cite[Corollary 3.4.10 and Remark 3.4.11, pp. 193-194]{LG2} (see also \cite[Proposition 1, p. 156]{MA}).

Suppose that $1<p<\infty$. Note that $T$ extends to a bounded operator from $\mathcal{H}^{(1,p)}$ into $(L^1,\ell^p)$, by Theorem \ref{theoremsing4}. Moreover, $(L^{\infty},\ell^{p'})\subset L^{p'}$, $1<p'<\infty$ and $T$ extends on $L^{p'}$ (by Theorem \ref{theoremsingaj}), and hence $T$ is extendable on $(L^{\infty},\ell^{p'})$. Another justification of this fact is to remark that $(L^{\infty},\ell^{p'})\subset L^{\infty}$ and $T$ is well defined on $L^{\infty}$ (see \cite[Remark 4.1.18, p. 223]{LG2}). Without loss generality, we can assume that $1\leq r<p'$ is such that $r'>\max\left\{2,p\right\}$.  Now, consider a family $\left\{\Omega^j\right\}_{j\in\mathbb{Z}}$ of open subsets of $\mathbb{R}^d$ such that $\left\|\sum_{j\in\mathbb{Z}}2^{j\eta}\chi_{\Omega^j}\right\|_{\frac{1}{\eta},\frac{p}{\eta}}<\infty$. We have to show that 
\begin{eqnarray}
\sum_{j\in\mathbb{Z}}2^j\textit{O}(T(f),\Omega^j,r)\leq C\left\|f\right\|_{\infty,p'}\left\|\sum_{j\in\mathbb{Z}}2^{j\eta}\chi_{\Omega^j}\right\|_{\frac{1}{\eta},\frac{p}{\eta}}^{\frac{1}{\eta}}, \label{01applicadualqpconv}
\end{eqnarray}
where $C>0$ is a constant independent of $f$.

Consider the dual operator $T^\ast$ of $T$. The kernel $K^\ast$ associated to $T^\ast$ is given by $K^\ast(x)=K(-x)$, for all $x\in\mathbb{R}^d$, and $$\int_{\mathbb{R}^d}T(f)(x)g(x)dx=\int_{\mathbb{R}^d}f(x)T^\ast(g)(x)dx,$$ for all $f,g\in L^2$. Moreover, $T^\ast$ and $K^\ast$ satisfy the same kind of inequalities as do $T$ and $K$ (for more details, see \cite[Chap. 1, (35), p. 36 and Chap. 4, 4.1, pp. 155-156]{MA}). Hence $T^\ast$ extends to a bounded operator from $\mathcal{H}^{(1,p)}$ into $(L^1,\ell^p)$, according to Theorem \ref{theoremsing4}. 
 
Let $f\in(L^{\infty},\ell^{p'})$. Consider the subspace $\mathcal{H}_{fin}^{(1,p)}$ of $\mathcal{H}^{(1,p)}$ consisting of finite linear combinations of $(1,r',\delta)$-atoms. Then, for all elements $g$ of $\mathcal{H}_{fin}^{(1,p)}$, we have $g\in L^2\cap L^{p}$ (because $r'>\max\left\{2,p\right\}$ implies that $(1,r',\delta)$-atoms are $(1,2,\delta)$-atoms and $(1,p,\delta)$-atoms) and $T^\ast(g)\in(L^1,\ell^p)$ with $\left\|T^\ast(g)\right\|_{1,p}\leq C\left\|g\right\|_{\mathcal{H}^{(1,p)}}$, $C>0$ being a constant independent of $g$. Further, according to \cite[Chap. 1, (35), p. 36 or Chap. 4, 4.1, pp. 155-156]{MA}, 
\begin{align}
\int_{\mathbb{R}^d}T(f)(x)g(x)dx=\int_{\mathbb{R}^d}f(x)T^\ast(g)(x)dx, \label{020applicadualqpconv}
\end{align}
for all $g\in\mathcal{H}_{fin}^{(1,p)}$, since $f\in(L^{\infty},\ell^{p'})\subset L^{\infty}\cap L^{p'}$ and $g\in L^2\cap L^{p}$. Hence 
\begin{align}
\left|\int_{\mathbb{R}^d}T(f)(x)g(x)dx\right|&=\left|\int_{\mathbb{R}^d}f(x)T^\ast(g)(x)dx\right|\nonumber\\
&\leq\left\|f\right\|_{\infty,p'}\left\|T^\ast(g)\right\|_{1,p}\leq C\left\|f\right\|_{\infty,p'}\left\|g\right\|_{\mathcal{H}^{(1,p)}},\label{021applicadualqpconv}
\end{align}
for all $g\in\mathcal{H}_{fin}^{(1,p)}$. Consequently, the mapping $G_{T(f)}:\mathcal{H}_{fin}^{(1,p)}\ni g\mapsto\int_{\mathbb{R}^d}T(f)(x)g(x)dx$ extends to a unique continuous linear functional $\widetilde{G_{T(f)}}$ on $\mathcal{H}^{(1,p)}$, with  
\begin{align}
\left\|\widetilde{G_{T(f)}}\right\|:=\left\|\widetilde{G_{T(f)}}\right\|_{(\mathcal{H}^{(1,p)})^{\ast}}=\sup_{\underset{g\in\mathcal{H}^{(1,p)}}{\left\|g\right\|_{\mathcal{H}^{(1,p)}}}\leq1}\left|G_{T(f)}(g)\right|\leq C\left\|f\right\|_{\infty,p'}. \label{2applicadualqpconv}
\end{align}
Furthermore, since $\widetilde{G_{T(f)}}\in(\mathcal{H}^{(1,p)})^{\ast}$ and 
\begin{align}
\widetilde{G_{T(f)}}(g)=\int_{\mathbb{R}^d}T(f)(x)g(x)dx, \label{022applicadualqpconv}
\end{align}
for all $g\in\mathcal{H}_{fin}^{(1,p)}$, with $T(f)\in L_{\mathrm{loc}}^r$ (since $T(f)\in L^{p'}\subset L_{\mathrm{loc}}^{p'}$ and $1\leq r<p'$), by repeating the second part of the proof of Theorem \ref{theoremdualqp} (see \cite[Theorem 3.7]{AbFt3}) with $\widetilde{G_{T(f)}}$, $T(f)$ and $g$ respectively to the place of $T$, $g$ and $f$, we get  
\begin{align}
\sum_{j\in\mathbb{Z}}2^j\textit{O}(T(f),\Omega^j,r)\leq C\left\|\widetilde{G_{T(f)}}\right\|\left\|\sum_{j\in\mathbb{Z}}2^{j\eta}\chi_{\Omega^j}\right\|_{\frac{1}{\eta},\frac{p}{\eta}}^{\frac{1}{\eta}}, \label{023applicadualqpconv}
\end{align}
with $C>0$ a constant independent of $T(f)$. It follows from (\ref{2applicadualqpconv}) and (\ref{023applicadualqpconv}) that  
\begin{align*}
\sum_{j\in\mathbb{Z}}2^j\textit{O}(T(f),\Omega^j,r)\leq C\left\|f\right\|_{\infty,p'}\left\|\sum_{j\in\mathbb{Z}}2^{j\eta}\chi_{\Omega^j}\right\|_{\frac{1}{\eta},\frac{p}{\eta}}^{\frac{1}{\eta}}, 
\end{align*}
which establishes (\ref{01applicadualqpconv}). Hence $T(f)\in\mathcal{L}_{r,\phi_{1},\delta}^{(1,p,\eta)}$ with $\left\|T(f)\right\|_{\mathcal{L}_{r,\phi_{1},\delta}^{(1,p,\eta)}}\leq C\left\|f\right\|_{\infty,p'}$, which completes the proof. 
\end{proof}

\begin{remark}
We point out that the proof of Theorem \ref{applicadualqpconv2} can be given without distinguishing the cases $p=1$ and $p>1$. For this approach, we have not to show Inequality (\ref{01applicadualqpconv}), which is not necessarily meaningful when $p=1$, since for $p=1$, it is not clear that $T(f)\in L_{\mathrm{loc}}^r$, necessary condition to the definition of $\textit{O}(T(f),\Omega^j,r)$, $j\in\mathbb{Z}$. However, Relations (\ref{020applicadualqpconv}),  (\ref{021applicadualqpconv}), (\ref{2applicadualqpconv}) and (\ref{022applicadualqpconv}) are valid for $1\leq p<\infty$, except $T(f)$ is not necessarily in $L^{\infty}$ for $p=1$. Thus overcoming the problem of Inequality (\ref{01applicadualqpconv}), once to Relation (\ref{022applicadualqpconv}), we appeal to Theorem 2.7 (2) (see \cite[Theorem 3.8 (2)]{AbFt3}), which allows to claim that there exists a function $h\in\mathcal{L}_{r,\phi_{1},\delta}^{(1,p,\eta)}$ such that $\widetilde{G_{T(f)}}=\left(\widetilde{G_{T(f)}}\right)_h$; namely $\widetilde{G_{T(f)}}(g)=\int_{\mathbb{R}^d}h(x)g(x)dx$, for all $g\in\mathcal{H}_{fin}^{(1,p)}$, with $\left\|h\right\|_{\mathcal{L}_{r,\phi_{1},\delta}^{(1,p,\eta)}}\leq C\left\|\widetilde{G_{T(f)}}\right\|$. Hence, according to Relation (\ref{022applicadualqpconv}), we can claim that $T(f)$ can be identified with the function $h$ so that $\left\|T(f)\right\|_{\mathcal{L}_{r,\phi_{1},\delta}^{(1,p,\eta)}}=\left\|h\right\|_{\mathcal{L}_{r,\phi_{1},\delta}^{(1,p,\eta)}}\leq C\left\|\widetilde{G_{T(f)}}\right\|$. Finally, we obtain $\left\|T(f)\right\|_{\mathcal{L}_{r,\phi_{1},\delta}^{(1,p,\eta)}}\leq C\left\|f\right\|_{\infty,p'}$, by (\ref{2applicadualqpconv}).
\end{remark}

\begin{remark}
In Theorem \ref{applicadualqpconv2} (\ref{bisapplicattheo15}), the positive integer $\delta$ can be replaced by $0$. 
In fact, under the assumptions of Theorem \ref{applicadualqpconv2}, $\mathcal{L}_{r,\phi_{1},0}^{(1,p,\eta)}\cong\mathcal{L}_{r,\phi_{1},\delta}^{(1,p,\eta)}$ for $\delta>0$, by Theorem 2.7 (see \cite[Theorem 3.8]{AbFt3}).
\end{remark}

\begin{cor}\label{applicadualqpconv3}
The Riesz transforms $R_j$, $1\leq j\leq d$, are bounded from $(L^{\infty},\ell^{p'})$ into $\mathcal{L}_{r,\phi_{1},\delta}^{(1,p,\eta)}$, for $1\leq p<\infty$, $\delta\geq0$, $1\leq r<p'$ and, $0<\eta<1$ if $1<r<p'$ or $0<\eta\leq1$ if $r=1$.
\end{cor}

\begin{thm}\label{applicadualqpconv4general}
Let $K$ be a tempered distribution in $\mathbb{R}^d$ that coincides with a locally integrable function on $\mathbb{R}^d\backslash\left\{0\right\}$ and satisfies assumptions (\ref{applicattheo12}) and (\ref{applicattheo15}). Suppose that $\frac{d}{d+\delta}<q\leq1<p<\infty$. Let $1\leq r<p'$ and, $0<\eta<q$ if $1<r<p'$ or $0<\eta\leq1$ if $r=1$. 

Then the operator $T(f)=K\ast f$, for all $f\in\mathcal{S}$, is extendable on $\mathcal{L}_{r,\phi_{1},\delta}^{(q,p,\eta)}$ and there exists a constant $C>0$ such that 
\begin{align}
\left\|T(f)\right\|_{\mathcal{L}_{r,\phi_{1},\delta}^{(q,p,\eta)}}\leq C\left\|f\right\|_{\mathcal{L}_{r,\phi_{1},\delta}^{(q,p,\eta)}},  \label{0applicadualqpconv4general0}
\end{align}
for all $f\in\mathcal{L}_{r,\phi_{1},\delta}^{(q,p,\eta)}$.  
\end{thm}
\begin{proof} 
Note that $T$ extends to a bounded operator from $\mathcal{H}^{(q,p)}$ into $\mathcal{H}^{(q,p)}$, by Theorem \ref{theoremsing5}. Moreover, $\mathcal{L}_{r,\phi_{1},\delta}^{(q,p,\eta)}\subset L^{p'}$, $1<p'<\infty$, by Remark \ref{remarqedualeqp1}, and $T$ extends on $L^{p'}$ (by Theorem \ref{theoremsingaj}), and hence $T$ is extendable on $\mathcal{L}_{r,\phi_{1},\delta}^{(q,p,\eta)}$. Without loss generality, we can assume that $1\leq r<p'$ is such that $r'>\max\left\{2,p\right\}$.  Now, consider a family $\left\{\Omega^j\right\}_{j\in\mathbb{Z}}$ of open subsets of $\mathbb{R}^d$ such that $\left\|\sum_{j\in\mathbb{Z}}2^{j\eta}\chi_{\Omega^j}\right\|_{\frac{1}{\eta},\frac{p}{\eta}}<\infty$. We have to show that 
\begin{align}
\sum_{j\in\mathbb{Z}}2^j\textit{O}(T(f),\Omega^j,r)\leq C\left\|f\right\|_{\mathcal{L}_{r,\phi_{1},\delta}^{(q,p,\eta)}}\left\|\sum_{j\in\mathbb{Z}}2^{j\eta}\chi_{\Omega^j}\right\|_{\frac{q}{\eta},\frac{p}{\eta}}^{\frac{1}{\eta}}, \label{3applicadualqpconvgener}
\end{align}
where $C>0$ is a constant independent of $f$.

Consider the dual operator $T^\ast$ of $T$. Then $T^\ast$ extends to a bounded operator from $\mathcal{H}^{(q,p)}$ to itself, by Theorem \ref{theoremsing5}.  
 
Let $f\in\mathcal{L}_{r,\phi_{1},\delta}^{(q,p,\eta)}$. Consider the subspace $\mathcal{H}_{fin}^{(q,p)}$ of $\mathcal{H}^{(q,p)}$ consisting of finite linear combinations of $(q,r',\delta)$-atoms. Then, for all elements $g$ of $\mathcal{H}_{fin}^{(q,p)}$, we have $T^\ast(g)\in\mathcal{H}^{(q,p)}$ with $\left\|T^\ast(g)\right\|_{\mathcal{H}^{(q,p)}}\leq C\left\|g\right\|_{\mathcal{H}^{(q,p)}}$, $C>0$ being a constant independent of $g$. Moreover, for all $g\in\mathcal{H}_{fin}^{(q,p)}$, we have $$\int_{\mathbb{R}^d}T(f)(x)g(x)dx=\int_{\mathbb{R}^d}f(x)T^\ast(g)(x)dx,$$ according to \cite[Chap. 1, (35), p. 36]{MA}, because $f\in\mathcal{L}_{r,\phi_{1},\delta}^{(q,p,\eta)}\subset L^{p'}$ (by Remark \ref{remarqedualeqp1}) and $g\in L^{p}$ (since $(q,r',\delta)$-atoms are also $(q,p,\delta)$-atoms given that $r'>\max\left\{2,p\right\}$), and hence $T(f)\in L^{p'}$ and $T^\ast(g)\in L^{p}$, and 
\begin{align}
\left|\int_{\mathbb{R}^d}f(x)T^\ast(g)(x)dx\right|\leq C\left\|f\right\|_{\mathcal{L}_{r,\phi_{1},\delta}^{(q,p,\eta)}}\left\|T^\ast(g)\right\|_{\mathcal{H}^{(q,p)}}, \label{4applicadualqpconvgener}
\end{align}
with $C>0$ a constant independent of $f$ and $g$ (we admit for the moment this inequality). Hence   
\begin{align}
\left|\int_{\mathbb{R}^d}T(f)(x)g(x)dx\right|&=\left|\int_{\mathbb{R}^d}f(x)T^\ast(g)(x)dx\right| \nonumber\\
&\leq C\left\|f\right\|_{\mathcal{L}_{r,\phi_{1},\delta}^{(q,p,\eta)}}\left\|T^\ast(g)\right\|_{\mathcal{H}^{(q,p)}}\leq C\left\|f\right\|_{\mathcal{L}_{r,\phi_{1},\delta}^{(q,p,\eta)}}\left\|g\right\|_{\mathcal{H}^{(q,p)}} \nonumber, 
\end{align}
for all $g\in\mathcal{H}_{fin}^{(q,p)}$. Consequently, the mapping $G_{T(f)}:\mathcal{H}_{fin}^{(q,p)}\ni g\mapsto\int_{\mathbb{R}^d}T(f)(x)g(x)dx$ extends to a unique continuous linear functional $\widetilde{G_{T(f)}}$ on $\mathcal{H}^{(q,p)}$, with  
\begin{align}
\left\|\widetilde{G_{T(f)}}\right\|:=\left\|\widetilde{G_{T(f)}}\right\|_{(\mathcal{H}^{(q,p)})^{\ast}}=\sup_{\underset{g\in\mathcal{H}^{(q,p)}}{\left\|g\right\|_{\mathcal{H}^{(q,p)}}}\leq1}\left|G_{T(f)}(g)\right|\leq C\left\|f\right\|_{\mathcal{L}_{r,\phi_{1},\delta}^{(q,p,\eta)}}. \label{5applicadualqpconvgener}
\end{align}
Furthermore, since $\widetilde{G_{T(f)}}\in(\mathcal{H}^{(q,p)})^{\ast}$, $$\widetilde{G_{T(f)}}(g)=\int_{\mathbb{R}^d}T(f)(x)g(x)dx,$$ for all $g\in\mathcal{H}_{fin}^{(q,p)}$, and $T(f)\in L_{\mathrm{loc}}^r$ (because $r<p'$ implies that $T(f)\in L^{p'}\subset L_{\mathrm{loc}}^{p'}\subset L_{\mathrm{loc}}^r$), by repeating the second part of the proof of Theorem \ref{theoremdualqp}, with $\widetilde{G_{T(f)}}$, $T(f)$ and $g$ respectively to the place of $T$, $g$ and $f$, we get 
\begin{align*}
\sum_{j\in\mathbb{Z}}2^j\textit{O}(T(f),\Omega^j,r)\leq C\left\|\widetilde{G_{T(f)}}\right\|\left\|\sum_{j\in\mathbb{Z}}2^{j\eta}\chi_{\Omega^j}\right\|_{\frac{q}{\eta},\frac{p}{\eta}}^{\frac{1}{\eta}}.
\end{align*}
It follows that  
\begin{align*}
\sum_{j\in\mathbb{Z}}2^j\textit{O}(T(f),\Omega^j,r)\leq C\left\|f\right\|_{\mathcal{L}_{r,\phi_{1},\delta}^{(q,p,\eta)}}\left\|\sum_{j\in\mathbb{Z}}2^{j\eta}\chi_{\Omega^j}\right\|_{\frac{q}{\eta},\frac{p}{\eta}}^{\frac{1}{\eta}},
\end{align*}
by (\ref{5applicadualqpconvgener}). Hence $T(f)\in\mathcal{L}_{r,\phi_{1},\delta}^{(q,p,\eta)}$ with $\left\|T(f)\right\|_{\mathcal{L}_{r,\phi_{1},\delta}^{(q,p,\eta)}}\leq C\left\|f\right\|_{\mathcal{L}_{r,\phi_{1},\delta}^{(q,p,\eta)}}$.\\ 

The proof of Theorem \ref{applicadualqpconv4general} will be complete if we prove (\ref{4applicadualqpconvgener}). Since $f\in\mathcal{L}_{r,\phi_{1},\delta}^{(q,p,\eta)}$, we know that the mapping $G_f:\mathcal{H}_{fin}^{(q,p)}\ni g\mapsto\int_{\mathbb{R}^d}f(x)g(x)dx$ extends to a unique continuous linear functional $\widetilde{G_f}$ on $\mathcal{H}^{(q,p)}$, with 
\begin{align}
|\widetilde{G_f}(g)|=|G_f(g)|\leq C\left\|f\right\|_{\mathcal{L}_{r,\phi_{1},\delta}^{(q,p,\eta)}}\left\|g\right\|_{\mathcal{H}^{(q,p)}},  \label{050applicadualqpconvgener}
\end{align}
for all $g\in\mathcal{H}_{fin}^{(q,p)}$. We also know that, for all $g\in\mathcal{H}_{fin}^{(q,p)}$, $T^\ast(g)\in\mathcal{H}^{(q,p)}$. However, it is not clear that $T^\ast(g)\in\mathcal{H}_{fin}^{(q,p)}$, for all $g\in\mathcal{H}_{fin}^{(q,p)}$. Hence we can not write $$\widetilde{G_f}(T^\ast(g))=G_f(T^\ast(g))=\int_{\mathbb{R}^d}f(x)T^\ast(g)(x)dx,$$ for all $g\in\mathcal{H}_{fin}^{(q,p)}$, and deduce (\ref{4applicadualqpconvgener}), according to (\ref{050applicadualqpconvgener}). However, we claim that 
\begin{align}
\widetilde{G_f}(T^\ast(g))=\int_{\mathbb{R}^d}f(x)T^\ast(g)(x)dx, \label{051applicadualqpconvgener}
\end{align}
for all $g\in\mathcal{H}_{fin}^{(q,p)}$. To see (\ref{051applicadualqpconvgener}), let $g\in\mathcal{H}_{fin}^{(q,p)}$. We have $T^\ast(g)\in\mathcal{H}^{(q,p)}\cap L^p$ (because $g\in L^p$ and $T^\ast$ is bounded from $L^{p}$ to itself, but also from $\mathcal{H}^{(q,p)}$ to itself). Therefore, by the proof of \cite[Theorem 4.4, pp. 1916-1919]{AbFt}, there exist a family $\left\{\left(a_{j,n},Q_{j,n}\right)\right\}_{(j,n)\in\mathbb{Z}\times\mathbb{Z_{+}}}$ of elements of $\mathcal{A}(q,r',\delta)$ and a family of scalars $\left\{\lambda_{j,n}\right\}_{(j,n)\in\mathbb{Z}\times\mathbb{Z_{+}}}$ such that 
\begin{align}
T^\ast(g)=\sum_{j=-\infty}^{+\infty}\sum_{n\geq 0}\lambda_{j,n}a_{j,n} 
\label{applicattheoqp1gener}
\end{align} 
almost everywhere and in the sense of $\mathcal{H}^{(q,p)}$ (unconditionally). Furthermore,
\begin{align}
\sum_{j=-\infty}^{+\infty}\sum_{n\geq 0}|\lambda_{j,n}a_{j,n}f|\in L^1. \label{applicattheoqp5gener}
\end{align}
For the proof of (\ref{applicattheoqp5gener}), first we recall that, by construction (see the proof of \cite[Theorem 4.4, pp. 1916-1919]{AbFt}), $|\lambda_{j,n}a_{j,n}|\leq C_1 2^j$ almost everywhere, $\text{supp}(a_{j,n})\subset Q_{j,n}:=C_0 Q_{j,n}^{\ast}$ and $\underset{n\geq 0}\sum\chi_{Q_{j,n}^{\ast}}\leq K(d)$ with, for every $j\in\mathbb{Z}$, $\underset{n\geq 0}\bigcup Q_{j,n}^{\ast}=\mathcal{O}^j:=\left\{x\in\mathbb R^d:\mathcal M_{\mathcal F_N^0}(T^\ast(g))(x)>2^j\right\}$, $N\geq\max\left\{\left\lfloor \frac{d}{q}\right\rfloor, \left\lfloor \frac{d}{p}\right\rfloor\right\}+1$ being an integer and $\mathcal M_{\mathcal F_N^0}(T^\ast(g))$ is the radial grand maximal function of $T^\ast(g)$ (with respect to $\mathcal F_N$) (see \cite{AbFt}, p. 1907, for the definitions of $\mathcal F_N$ and $\mathcal M_{\mathcal F_N^0}(T^\ast(g))$). Thus, 
\begin{align*}
\left\|\sum_{j=-\infty}^{+\infty}\sum_{n\geq 0}|\lambda_{j,n}a_{j,n}f|\right\|_1&=\left\||f|\sum_{j=-\infty}^{+\infty}\sum_{n\geq 0}|\lambda_{j,n}a_{j,n}|\right\|_1\\
&\leq\left\|f\right\|_{L^{p'}}\left\|\sum_{j=-\infty}^{+\infty}\sum_{n\geq 0}|\lambda_{j,n}a_{j,n}|\right\|_{L^p}\\
&\leq C_1\left\|f\right\|_{L^{p'}}\left\|\sum_{j=-\infty}^{+\infty}\sum_{n\geq 0}2^j\chi_{Q_{j,n}}\right\|_{L^p}\\
&\leq C(\varphi,d,N,\delta)\left\|f\right\|_{L^{p'}}\left\|\sum_{j=-\infty}^{+\infty}\sum_{n\geq 0}2^j\left[\mathfrak{M}\left(\chi_{Q_{j,n}^{\ast}}\right)\right]^2\right\|_{L^p}\\
&=C(\varphi,d,N,\delta)\left\|f\right\|_{L^{p'}}\left\|\left(\sum_{j=-\infty}^{+\infty}\sum_{n\geq 0}2^j\left[\mathfrak{M}\left(\chi_{Q_{j,n}^{\ast}}\right)\right]^2\right)^{\frac{1}{2}}\right\|_{L^{2p}}^2\\
&\leq C(\varphi,d,p,N,\delta)\left\|f\right\|_{L^{p'}}\left\|\left(\sum_{j=-\infty}^{+\infty}\sum_{n\geq 0}2^j\left(\chi_{Q_{j,n}^{\ast}}\right)^2\right)^{\frac{1}{2}}\right\|_{L^{2p}}^2\\
&=C(\varphi,d,p,N,\delta)\left\|f\right\|_{L^{p'}}\left\|\sum_{j=-\infty}^{+\infty}\sum_{n\geq 0}2^j\chi_{Q_{j,n}^{\ast}}\right\|_{L^p}\\
&\leq C(\varphi,d,p,N,\delta)\left\|f\right\|_{L^{p'}}\left\|\sum_{j=-\infty}^{+\infty}2^j\chi_{\mathcal{O}^j}\right\|_{L^p}\\
&\leq C(\varphi,d,p,N,\delta)\left\|f\right\|_{L^{p'}}\left\|\mathcal M_{\mathcal F_N^0}(T^\ast(g))\right\|_{L^p},
\end{align*}
by H\"older inequality, \cite[Lemma 3.3]{AbFt3}, \cite[Theorem 1, p. 107]{CFES} and \cite[(4.18), p. 1919]{AbFt}, where $\mathfrak{M}$ denotes the Hardy-Littlewood maximal operator, defined for a locally integrable function $f$ by 
\begin{align*}
\mathfrak{M}(f)(x):=\underset{r>0}\sup\ |B(x,r)|^{-1}\int_{B(x,r)}|f(y)|dy,\ \ x\in\mathbb{R}^d. 
\end{align*}  
But, since $p>1$, $N\geq\max\left\{\left\lfloor\frac{d}{q}\right\rfloor,\left\lfloor\frac{d}{p}\right\rfloor\right\}+1=\left\lfloor\frac{d}{q}\right\rfloor+1>\left\lfloor\frac{d}{p}\right\rfloor+1$ and $T^\ast(g)\in L^{p}$, we have 
\begin{align*}
\left\|\mathcal M_{\mathcal F_N^0}(T^\ast(g))\right\|_{L^p}\approx\left\|T^\ast(g)\right\|_{L^p}<\infty, 
\end{align*}
by \cite[Remark, pp. 15-16]{MBOW}. Hence 
\begin{align*}
\left\|\sum_{j=-\infty}^{+\infty}\sum_{n\geq 0}|\lambda_{j,n}a_{j,n}f|\right\|_1\lsim\left\|f\right\|_{L^{p'}}\left\|T^\ast(g)\right\|_{L^p}<\infty,
\end{align*}
which states (\ref{applicattheoqp5gener}). From (\ref{applicattheoqp1gener}) and (\ref{applicattheoqp5gener}), it follows that 
\begin{align*}
\int_{\mathbb{R}^d}f(x)T^\ast(g)(x)dx&=\int_{\mathbb{R}^d}f(x)\left(\sum_{j=-\infty}^{+\infty}\sum_{n\geq 0}\lambda_{j,n}a_{j,n}(x)\right)dx\\
&=\int_{\mathbb{R}^d}\sum_{j=-\infty}^{+\infty}\sum_{n\geq 0}(\lambda_{j,n}a_{j,n}(x)f(x))dx\\
&=\sum_{j=-\infty}^{+\infty}\sum_{n\geq 0}\lambda_{j,n}\int_{\mathbb{R}^d}f(x)a_{j,n}(x)dx,
\end{align*}
by Fubini Theorem. Moreover, 
\begin{align*}
\sum_{j=-\infty}^{+\infty}\sum_{n\geq 0}\lambda_{j,n}\int_{\mathbb{R}^d}f(x)a_{j,n}(x)dx&=\sum_{j=-\infty}^{+\infty}\sum_{n\geq 0}\lambda_{j,n}G_f(a_{j,n})\\
&=\sum_{j=-\infty}^{+\infty}\sum_{n\geq 0}\lambda_{j,n}\widetilde{G_f}(a_{j,n})\\
&=\widetilde{G_f}\left(\sum_{j=-\infty}^{+\infty}\sum_{n\geq 0}\lambda_{j,n}a_{j,n}\right)=\widetilde{G_f}(T^\ast(g)),
\end{align*}
since $T^\ast(g)=\sum_{j=-\infty}^{+\infty}\sum_{n\geq 0}\lambda_{j,n}a_{j,n}$ (unconditionally) in $\mathcal{H}^{(q,p)}$ and $\widetilde{G_f}$ is a continuous linear functional on $\mathcal{H}^{(q,p)}$. Therefore, 
\begin{align*}
\int_{\mathbb{R}^d}f(x)T^\ast(g)(x)dx=\widetilde{G_f}(T^\ast(g)),
\end{align*}
which establishes (\ref{051applicadualqpconvgener}). 

It follows that 
\begin{align*}
\left|\int_{\mathbb{R}^d}f(x)T^\ast(g)(x)dx\right|&=\left|\widetilde{G_f}(T^\ast(g))\right|\\
&\leq\left\|\widetilde{G_f}\right\|\left\|T^\ast(g)\right\|_{\mathcal{H}^{(q,p)}}\leq C\left\|f\right\|_{\mathcal{L}_{r,\phi_{1},\delta}^{(q,p,\eta)}}\left\|T^\ast(g)\right\|_{\mathcal{H}^{(q,p)}},
\end{align*}
which establishes (\ref{4applicadualqpconvgener}), and hence completes the proof of Theorem \ref{applicadualqpconv4general}.
\end{proof}

\begin{remark}\label{0applicadualqpconv4generalbis0}
In Theorem \ref{applicadualqpconv4general} (\ref{0applicadualqpconv4general0}), the positive integer $\delta$ can be replaced by $0$ provided that $\left\lfloor d\left(\frac{1}{q}-1\right)\right\rfloor=0$. In fact, under the assumptions of Theorem \ref{applicadualqpconv4general}, $\mathcal{L}_{r,\phi_{1},0}^{(q,p,\eta)}\cong\mathcal{L}_{r,\phi_{1},\delta}^{(q,p,\eta)}$ for $\delta>0$ provided that $\left\lfloor d\left(\frac{1}{q}-1\right)\right\rfloor=0$, by Theorem 2.7 (see \cite[Theorem 3.8]{AbFt3}).
\end{remark}

\begin{cor}\label{applicadualqpconv4generalbis}
Let $K$ be a tempered distribution in $\mathbb{R}^d$ that coincides with a locally integrable function on $\mathbb{R}^d\backslash\left\{0\right\}$ and satisfies assumptions (\ref{applicattheo12}) and (\ref{applicattheo15}). Suppose that $\frac{d}{d+\delta}<q\leq1$ and $q\leq p<\infty$. Let $\max\left\{1,p\right\}<r\leq\infty$ and, $0<\eta<q$ if $r<\infty$ or $0<\eta\leq1$ if $r=\infty$. 

Then the operator $T(f)=K\ast f$, for all $f\in\mathcal{S}$, is extendable on $\mathcal{L}_{r',\phi_{1},\delta}^{(q,p,\eta)}$ and there exists a constant $C>0$ such that
\begin{eqnarray*}
\left\|T(f)\right\|_{\mathcal{L}_{r',\phi_{1},\delta}^{(q,p,\eta)}}\leq C\left\|f\right\|_{\mathcal{L}_{r',\phi_{1},\delta}^{(q,p,\eta)}},
\end{eqnarray*}
for all $f\in\mathcal{L}_{r',\phi_{1},\delta}^{(q,p,\eta)}$.  
\end{cor}
\begin{proof}
We distinguish the cases $p\leq1$ and $p>1$. The case $p>1$ is merely Theorem \ref{applicadualqpconv4general}. When $p\leq1$, we have $\mathcal{L}_{r',\phi_{1},\delta}^{(q,p,\eta)}=\mathcal{L}_{r',\phi_{1},\delta}$ with equivalent norms, by (\ref{dualqp3bisbis}), and $T$ is bounded from $\mathcal{L}_{r',\phi_{1},\delta}$ to itself, according to \cite{PeeJ}.
\end{proof}

\begin{remark}
Remark \ref{0applicadualqpconv4generalbis0} is valid for Corollary \ref{applicadualqpconv4generalbis} provided that $\left\lfloor d\left(\frac{1}{q}-1\right)\right\rfloor=0$. Also, when $p\leq1$, the assumption $q\leq p$ is not needed and we can take $0<\eta\leq1$ for $1<r\leq\infty$.
\end{remark}

\begin{cor}\label{applicadualqpconv5general}
The Riesz transforms $R_j$, $1\leq j\leq d$, are bounded from $\mathcal{L}_{r,\phi_{1},\delta}^{(q,p,\eta)}$ into $\mathcal{L}_{r,\phi_{1},\delta}^{(q,p,\eta)}$, for $0<q\leq1$, $q\leq p<\infty$, $\delta\geq\left\lfloor d\left(\frac{1}{q}-1\right)\right\rfloor$ and, $1\leq r<p'$ if $1<p$ or $1\leq r<\infty$ if $p\leq1$, with $0<\eta<q$ if $1<r$ or $0<\eta\leq1$ if $r=1$.
\end{cor}

\begin{remark}
In Corollary \ref{applicadualqpconv5general}, when $p\leq1$, the assumption $q\leq p$ is not needed and we can take $0<\eta\leq1$ for $1<r\leq\infty$. 
\end{remark}

\subsection{Calder\'on-Zygmund operator} 
 
Let $\triangle:=\left\{(x,x):\ x\in\mathbb{R}^d\right\}$ be the diagonal of $\mathbb{R}^d\times\mathbb{R}^d$.
We say that a function $K:\mathbb{R}^d\times\mathbb{R}^d\backslash\triangle\rightarrow\mathbb{C}$ is a standard kernel if there exist a constant $A>0$ and an exponent $\mu>0$ such that:
\begin{eqnarray}
|K(x,y)|\leq A|x-y|^{-d}; \label{integralsing1}
\end{eqnarray}
\begin{equation}
|K(x,y)-K(x,z)|\leq A\frac{|y-z|^{\mu}}{|x-y|^{d+\mu}}\ ,\ \text{ if }\ \ |x-y|\geq 2|y-z| \label{integralsing2}
\end{equation}
and
\begin{equation}
|K(x,y)-K(w,y)|\leq A\frac{|x-w|^{\mu}}{|x-y|^{d+\mu}}\ ,\ \text{ if }\ \ |x-y|\geq 2|x-w|. \label{integralsing3}
\end{equation}

We denote by $\mathcal{SK}(\mu,A)$ the class of all standard kernels $K$ with exponent and constant $\mu$ and $A$. 

\begin{defn}\cite[Definition 5.11]{JD} An operator $T$ is a (generalized) Calder\'on-Zygmund operator if 
\begin{enumerate}
\item $T$ is bounded on $L^2$;
\item There exists a standard kernel $K$ such that for $f\in L^2$ with compact support,
\begin{eqnarray}
T(f)(x)=\int_{\mathbb{R}^d}K(x,y)f(y)dy,\ x\notin\text{supp}(f). \label{egopcalzyg1}
\end{eqnarray}
\end{enumerate}
\end{defn}

We proved the following result in \cite{AbFt1} (see \cite[Theorem 4.2]{AbFt1}).

\begin{thm} \label{theoremsing1}
Let $T$ be a Calder\'on-Zygmund operator with kernel $K\in \mathcal{SK}(\mu,A)$. If $\frac{d}{d+\mu}<q\leq 1$, then $T$ extends to a bounded operator from $\mathcal{H}^{(q,p)}$ to $(L^q,\ell^p)$. 
\end{thm}

Our main result is the following.

\begin{thm}\label{applicadualqp0}
Suppose that $1\leq p<\infty$. Let $\delta\geq0$ be an integer, $1\leq r<p'$ and, $0<\eta<1$ if $1<r<p'$ or $0<\eta\leq1$ if $r=1$. Let $T$ be a Calder\'on-Zygmund operator with kernel $K\in \mathcal{SK}(\mu,A)$. Then there exists a constant $C>0$ such that
\begin{eqnarray*}
\left\|T(f)\right\|_{\mathcal{L}_{r,\phi_{1},\delta}^{(1,p,\eta)}}\leq C\left\|f\right\|_{\infty,p'},
\end{eqnarray*}
for all $f\in(L^{\infty},\ell^{p'})$.
\end{thm}
\begin{proof}
Using Theorem \ref{theoremsing1}, the proof of Theorem \ref{applicadualqp0} is similar to the one of Theorem \ref{applicadualqpconv2}; the details are hence left to the reader.
\end{proof}

\end{document}